\documentclass[11pt,a4paper]{article}

\usepackage[utf8]{inputenc}
\usepackage[T1]{fontenc}
\usepackage[a4paper,margin=1in]{geometry}
\usepackage{lmodern}
\usepackage{microtype}
\usepackage{amsmath,amssymb,amsfonts,mathtools}
\usepackage{amsthm}
\usepackage{bm}
\usepackage{mathrsfs}
\usepackage{enumitem}
\usepackage{hyperref}
\usepackage{slashed} 
\usepackage[nameinlink,capitalize]{cleveref}
\DeclareMathOperator{\Pol}{Pol}
\DeclareMathOperator{\Ric}{Ric}
\DeclareMathOperator{\Scal}{Scal}
\newcommand{\Q}{\mathbb{Q}}
\newcommand{\Z}{\mathbb{Z}}

\newtheorem{theorem}{Theorem}[section]
\newtheorem{proposition}[theorem]{Proposition}
\newtheorem{lemma}[theorem]{Lemma}
\newtheorem{corollary}[theorem]{Corollary}
\theoremstyle{definition}
\newtheorem{definition}[theorem]{Definition}
\newtheorem{remark}[theorem]{Remark}
\newtheorem{example}[theorem]{Example}

\newcommand{\R}{\mathbb{R}}
\newcommand{\N}{\mathbb{N}}

\newcommand{\vol}{\mathrm{vol}}
\newcommand{\Fr}{\mathrm{Fr}}
\newcommand{\Met}{\mathrm{Met}}
\newcommand{\Diff}{\mathrm{Diff}}

\newcommand{\Ahat}{\widehat{A}}

\newcommand{\End}{\mathrm{End}}
\newcommand{\T}{\mathrm{T}}
\newcommand{\Tstar}{\mathrm{T}^*}
\newcommand{\GammaSec}{\Gamma}

\title{\Large Universal Bundles of Metrics and Polymetrics:\\
A  Generalization}
\author{%
  Shouvik Datta Choudhury\\
  \small GapCrud Private Limited (Capsule Labs)\\
  \small \texttt{shouvikdc8645@gmail.com}
}
\date{\normalsize \today}

\begin{document}
\maketitle

\begin{abstract}
We formalize the ``metric bundle'' viewpoint by defining, for any smooth $n$--manifold $M$, the open fiberwise cones $\mathcal{G}^{p,q}\subset S^2\Tstar M$ of nondegenerate symmetric bilinear forms with fixed signature $(p,q)$, and we package \emph{multi\-metric} (``polymetric'') geometries as sections of finite products of such cones. This framework subsumes Riemannian and pseudo-Riemannian metrics, and admits clean extensions to conformal, densitized, Finsler, and sub-Riemannian structures. It also interfaces correctly with families index theory (Atiyah--Singer), equivariant/groupoid settings, and coarse/KK-theory. On compact $M$ the Riemannian space of metrics is convex (hence contractible), giving a transparent base for moduli and deformation theory. Precise statements and proofs are provided beginning with the foundational objects below.
\end{abstract}

\section{Introduction and the foundational objects}
Throughout, $M$ denotes a smooth, second countable, paracompact manifold of dimension $n$ (boundary allowed unless stated otherwise). Write $S^2\Tstar M\to M$ for the bundle of symmetric covariant $2$-tensors and $\GammaSec(\,\cdot\,)$ for smooth sections.
\subsection*{Spaces of metrics, slices, and moduli}
The Fr\'echet (or Sobolev) manifold structure of the space $\Met(M)$ of Riemannian metrics and the action of
$\Diff(M)$ are foundational to deformation and moduli problems. The classical Ebin--Palais slice furnishes a local
model of the quotient $\Met(M)/\Diff(M)$ via a divergence/Bianchi-type gauge, giving immersed orbits and
a transverse slice \cite{Ebin,Palais}. In the Einstein context, Besse’s monograph \cite{Besse} is the standard
reference; early deformation theory of curvature functionals, scalar curvature constraints, and Einstein equations
was developed by Fischer--Marsden \cite{FischerMarsden} and Koiso \cite{Koiso1,Koiso2}. The Lichnerowicz Laplacian
and its ellipticity properties on symmetric 2-tensors underpin the Fredholm mapping properties needed for Kuranishi
models. For constant scalar curvature (CSC) problems, the Yamabe operator and its variational structure feature
prominently (see \cite{Hebey} for analysis on manifolds). These tools pass componentwise to our ``polymetric'' product
slices, with the shared $\Diff(M)$-gauge enforcing a coupled structure across multiple metrics.

\subsection*{Families index theory and superconnections}
The Atiyah--Singer index theorem \cite{AS1,AS2,AS3} and its families form (pushforward in $K$-theory) provide the
topological backbone for parameterized elliptic problems. Quillen’s superconnections introduced a powerful formalism
that Bismut elevated to a complete analytic proof of the families index formula via heat-kernel methods; Bismut’s
original papers and the monograph of Berline--Getzler--Vergne (BGV) \cite{BGV} are definitive accounts.
Bismut--Freed \cite{BismutFreed1,BismutFreed2} constructed determinant line bundles and analyzed their curvature and
holonomy. For manifolds with boundary, Bismut--Cheeger \cite{BC-eta,BC-I,BC-II} developed the families index
with APS-type boundary conditions, adiabatic limits of $\eta$-invariants, and explicit Chern character formulas;
these are exactly the transgression and invariance statements invoked in our framework. Equivariant families are
treated by the fixed-point formulas of Atiyah--Segal--Singer \cite{AtiyahSegalSinger} and Berline--Vergne
\cite{BerlineVergne}, refined analytically in the superconnection formalism.

\subsection*{Orbifolds (V-manifolds), groupoids, and stacks}
Orbifold index theory originates with Kawasaki, who established the signature theorem, the Riemann--Roch theorem,
and the general index theorem for V-manifolds \cite{KawasakiSign,KawasakiRR,KawasakiIndex}. The appearance of
twisted-sector (inertia orbifold) contributions in characteristic forms is now standard in the stacky language.
Groupoid $C^\ast$-algebras provide a flexible receptacle for indices of $\mathcal{G}$-equivariant elliptic families;
see Connes \cite{ConnesNCG} and Moerdijk--Mr\v{c}un \cite{MoerdijkMrcun} for the differential and operator-algebraic
machinery. Our polymetric sections are phrased $\mathcal{G}$-equivariantly, so slice theorems, families indices,
and characteristic classes carry over verbatim.

\subsection*{Foliations and longitudinal indices}
Longitudinal (leafwise) elliptic operators on a foliated manifold $(M,\mathcal{F})$ have index classes in the
$K$-theory of the holonomy groupoid $C^\ast$-algebra. Connes--Skandalis \cite{ConnesSkandalis} identified the
longitudinal index and its Chern--Connes character; secondary invariants (e.g.\ Godbillon--Vey) arise through
pairings with cyclic cocycles. Again, our polymetric layer is leafwise, so the gauge-fixing and Fredholm properties
follow from the same elliptic arguments used in the Einstein/CSC deformations.

\subsection*{Coarse geometry, Roe algebras, and $KK$-theory}
Roe’s coarse index theory \cite{RoeIndex,RoeElliptic} attaches $K$-theory classes to Dirac-type operators on complete
manifolds of bounded geometry; the classes are invariant under quasi-isometries and controlled perturbations.
Higson--Roe \cite{HigsonRoe} give a comprehensive account linking coarse indices, assembly maps, and large-scale
$K$-homology; Kasparov’s $KK$-theory \cite{Kasparov} supplies the categorical backbone for functorial products and
pushforwards. Partitioned and relative index theorems \cite{RoePartitioned} and Callias-type deformations
\cite{Anghel,Callias} provide computational leverage on noncompact spaces. Our coarse statements adapt these results
to polymetrics by enforcing uniform bounded-geometry control componentwise and using equivalence at infinity.

\subsection*{Finsler and sub-Riemannian geometries via cone constructions}
The cone viewpoint generalizes Riemannian metrics to Finsler norms on $T^\ast M$ and sub-Riemannian structures on
subbundles. Standard references include Bao--Chern--Shen \cite{BaoChernShen} for Finsler geometry and Montgomery
\cite{Montgomery} for sub-Riemannian control-geometric foundations; Binet--Legendre/John-ellipsoid approximations
yield bi-Lipschitz Riemannian approximants compatible with our multi-metric analytic framework.

\subsection*{Bounded geometry, Sobolev scales, and elliptic regularity}
Uniform analytic control under bounded geometry is treated in detail by Shubin \cite{Shubin}, Eichhorn \cite{Eichhorn},
and in PDE texts by Taylor \cite{TaylorPDE}. These results ensure essential self-adjointness, uniform Sobolev
equivalences, and heat-kernel estimates needed throughout (e.g.\ in the coarse section and in superconnection
transgression bounds).

\begin{definition}[Universal bundle of metrics of fixed signature]
For integers $p,q\ge 0$ with $p+q=n$, define the open subbundle
\[
\mathcal{G}^{p,q}\;:=\;\bigl\{\,\beta\in S^2\Tstar M \;\big|\; \beta_x \text{ is nondegenerate of inertia }(p,q)\text{ for all }x\in M \,\bigr\}.
\]
A \emph{(pseudo-)Riemannian metric of signature $(p,q)$} on $M$ is a section $g\in \GammaSec(\mathcal{G}^{p,q})$. In particular, $\mathcal{G}^+:=\mathcal{G}^{n,0}$ is the bundle of Riemannian metrics and $\Met(M):=\GammaSec(\mathcal{G}^+)$.
\end{definition}

\begin{remark}[Associated-bundle description]
Let $\Fr(\T M)\to M$ be the $GL(n)$-frame bundle. Fiberwise, the set of bilinear forms of signature $(p,q)$ is diffeomorphic to $GL(n)/O(p,q)$. Consequently,
\[
\mathcal{G}^{p,q}\;\cong\; \Fr(\T M)\times_{GL(n)} \bigl(GL(n)/O(p,q)\bigr),
\]
so a section of $\mathcal{G}^{p,q}$ is equivalent to a reduction of $\Fr(\T M)$ to $O(p,q)$ (a $G$-structure).
\end{remark}

\begin{proposition}[Openness and convexity (Riemannian case)]
\label{prop:convex}
$\mathcal{G}^{p,q}\subset S^2\Tstar M$ is an open subbundle. For $(p,q)=(n,0)$ and $g_0,g_1\in \Met(M)$, the convex path $g_t:=(1-t)g_0+t g_1$ lies in $\Met(M)$ for all $t\in[0,1]$. Hence $\Met(M)$ is convex (in particular, contractible) in the Fr\'echet topology of $\GammaSec(S^2\Tstar M)$.
\end{proposition}

\begin{definition}[Polymetrics: rigorous ``multi-norms'']
Fix a finite index set $I$ and signatures $\{(p_i,q_i)\}_{i\in I}$. The \emph{polymetric bundle} is the fiber product
\[
\mathcal{G}^{\vec p,\vec q}\;:=\;\prod_{i\in I}\mathcal{G}^{p_i,q_i}\;\longrightarrow\; M,
\]
and a \emph{polymetric} on $M$ is a section $\mathbf{g}=(g_i)_{i\in I}\in \GammaSec\bigl(\mathcal{G}^{\vec p,\vec q}\bigr)$. When each $(p_i,q_i)=(n,0)$ this encodes a \emph{multi-Riemannian} structure. Constraints (e.g.\ Ricci bounds, volume normalization) define closed subspaces of sections via smooth equalities/inequalities.
\end{definition}

\begin{remark}[Levi--Civita layer and beyond]
For $g\in\Met(M)$ the torsion-free $g$-compatible connection $\nabla^g$ exists and is unique (Koszul formula). The framework further accommodates conformal classes $\mathcal{C}=\mathcal{G}^+/\R_{>0}$, densitized metrics $\mathcal{G}^+\otimes (\Lambda^n\Tstar M)^{\otimes w}$, and (by analogous cone constructions in $\Tstar M$ and subbundles of $\T M$) Finsler and sub-Riemannian data.
\end{remark}

\bigskip
\noindent\textbf{Guide to the paper.}
\begin{itemize}[leftmargin=1.4em]
  \item \S2 develops analytic and topological properties of $\GammaSec(\mathcal{G}^{p,q})$ and polymetric spaces.
  \item \S3 treats moduli: the stack $[\GammaSec(\mathcal{G}^{p,q})/\Diff(M)]$ and deformation slices.
  \item \S4 connects families of metrics to the Atiyah--Singer families index ($K$-theory on parameter spaces) with the correct characteristic classes.
  \item \S5 extends to equivariant, groupoid/orbifold, and foliation settings; \S6 to coarse/KK-theory.
\end{itemize}
\section{Topology and analysis on spaces of metrics and polymetrics}
\label{sec:topology-analysis}

We work in the convenient/Fr\'echet category for spaces of smooth sections. For a finite-rank smooth vector bundle
$E\to M$, write $\Gamma^\infty(E)$ for the space of smooth sections endowed with the standard Fr\'echet topology
induced by seminorms
\[
\|s\|_{C^k(K)}:=\sum_{j=0}^k \sup_{x\in K}\bigl|\nabla^{j}s(x)\bigr|,
\qquad k\in\mathbb{N},\; K\Subset M\text{ compact,}
\]
where $\nabla$ is any auxiliary connection on $E$ and the norm is induced by any auxiliary background Riemannian metric
on $M$; different choices yield equivalent families of seminorms.

\subsection{Open metric cones and stability of signature}

\begin{lemma}[Openness of the signature locus]
\label{lem:open-signature}
For fixed $(p,q)$ with $p+q=n$, the subset $\mathcal{G}^{p,q}\subset S^2\Tstar M$ is an open subbundle.
Consequently, the space of sections $\Gamma(\mathcal{G}^{p,q})\subset \Gamma(S^2\Tstar M)$ is open in the
Fr\'echet topology.
\end{lemma}

\begin{proof}
Fix $x\in M$ and a symmetric bilinear form $\beta_x$ of inertia $(p,q)$.
By Sylvester's law of inertia, the map $S^2T_x^*M\to \R^{n(n+1)/2}$ that sends $\beta$ to its matrix in a fixed frame
is continuous, and nondegeneracy with fixed inertia is characterized by the nonvanishing of finitely many principal
minors with prescribed signs; this is an open condition. Smoothness in $x$ then shows $\mathcal{G}^{p,q}$ is an open
submanifold of $S^2\Tstar M$. Openness of the section space follows from the general fact that sections into an open
subbundle form an open subset of the ambient Fr\'echet space of sections.
\end{proof}

\begin{proposition}[Stability under small perturbations]
\label{prop:stability}
Let $g\in \Gamma(\mathcal{G}^{p,q})$. There exists a $C^0$-neighborhood $U$ of $0$ in $\Gamma(S^2\Tstar M)$ such that
for all $h\in U$ one has $g+h\in \Gamma(\mathcal{G}^{p,q})$ (the inertia is preserved).
\end{proposition}

\begin{proof}
Work in a finite cover by coordinate charts and trivializations. On each compact chart $K$, the eigenvalues of
the matrix of $g_x$ vary continuously with $x$ and with the perturbation $h_x$.
Choose $\varepsilon>0$ smaller than half the minimum absolute value of eigenvalues of $g_x$ over $K$;
if $\|h\|_{C^0(K)}<\varepsilon$, then $g_x+h_x$ is nondegenerate and has the same number of positive/negative eigenvalues
as $g_x$. A partition of unity yields a global neighborhood $U$ with the stated property.
\end{proof}

\begin{corollary}[Convexity in the Riemannian case]
\label{cor:convex}
For $(p,q)=(n,0)$ and $g_0,g_1\in \Met(M)$, the straight path $g_t=(1-t)g_0+t g_1$ lies in $\Met(M)$ for all
$t\in[0,1]$. Hence $\Met(M)$ is convex and in particular contractible in the Fr\'echet topology.
\end{corollary}

\begin{proof}
Pointwise, the cone of positive-definite forms is convex. Smooth convex combinations remain smooth and positive-definite.
\end{proof}

\subsection{Smooth algebra of basic operations}

The following maps are smooth between Fr\'echet manifolds of sections:

\begin{theorem}[Basic smooth operations on metrics]
\label{thm:smooth-ops}
Fix a compact $M$. The assignments
\[
g \longmapsto g^{-1}\in \Gamma(S^2 TM),\qquad
g \longmapsto d\vol_g\in \Gamma(\Lambda^n \Tstar M),\qquad
g \longmapsto \Gamma(g)\in \Gamma\bigl(T^*M\otimes \End(TM)\bigr),
\]
\[
g \longmapsto \nabla^g\in \Gamma\bigl(T^*M\otimes \End(T^*M)\bigr),\qquad
g \longmapsto \mathrm{Riem}(g),\,\mathrm{Ric}(g),\,\mathrm{Scal}(g)
\]
are smooth in the Fr\'echet sense. Here $\Gamma(g)$ denotes Christoffel symbols via the Koszul formula, and
$\mathrm{Riem},\mathrm{Ric},\mathrm{Scal}$ are the curvature, Ricci, and scalar curvatures of $g$.
\end{theorem}

\begin{proof}
All displayed expressions are obtained from $g$ and its derivatives by a finite composition of algebraic operations
(addition, tensor product, contraction) and inversion $g\mapsto g^{-1}$ on the open cone of positive-definite symmetric
endomorphisms of $TM$. Inversion is smooth on the open subset of invertible endomorphisms, and the remaining
operations are continuous multilinear (hence smooth) in the Fr\'echet category. Compactness of $M$ avoids issues of
uniform control at infinity.
\end{proof}

\begin{remark}[Regularity in local norms]
For $k\ge 2$, the map $g\mapsto \Gamma(g)$ is $C^{k-1}$ when $g$ is measured in $C^k$-topology; similarly the curvature
tensors are $C^{k-2}$ in $g$. This is the usual loss of derivatives due to first/second derivatives in the formulas.
\end{remark}

\subsection{Polymetrics as products and induced function spaces}

\begin{definition}[Polymetric space and product Fr\'echet structure]
Let $I$ be a finite index set and $(p_i,q_i)$ signatures. The polymetric bundle
$\mathcal{G}^{\vec p,\vec q}=\prod_{i\in I}\mathcal{G}^{p_i,q_i}\to M$ is an open subbundle of
$\prod_{i\in I} S^2 \Tstar M$, and the space of polymetrics
\[
\Pol(M;\vec p,\vec q):=\Gamma\bigl(\mathcal{G}^{\vec p,\vec q}\bigr)
\]
is the open subset of the product Fr\'echet space $\prod_{i\in I}\Gamma(S^2\Tstar M)$ consisting of those tuples
with the prescribed inertias.
\end{definition}

\begin{lemma}[Smoothness of projections and constructions]
\label{lem:poly-smooth}
The coordinate projections $\Pol(M;\vec p,\vec q)\to \Gamma(\mathcal{G}^{p_i,q_i})$ are smooth submersions.
If each $(p_i,q_i)=(n,0)$, then the maps
\[
\mathbf{g}=(g_i)_{i\in I}\longmapsto \bigl(\nabla^{g_i}\bigr)_{i\in I},\quad
\mathbf{g}\longmapsto \bigl(\mathrm{Riem}(g_i),\mathrm{Ric}(g_i),\mathrm{Scal}(g_i)\bigr)_{i\in I}
\]
are smooth.
\end{lemma}

\begin{proof}
All statements follow from Theorem~\ref{thm:smooth-ops} and the functoriality of products in the Fr\'echet category.
\end{proof}

\begin{definition}[Multi--Sobolev norms]
\label{def:multi-sobolev}
Let $\mathbf{g}=(g_i)_{i\in I}$ be a multi-Riemannian polymetric (i.e.\ each $g_i$ is positive definite) and fix
$k\in\N$. For a tensor field $u$ on $M$ define
\[
\|u\|_{W^{k,2}(\mathbf{g})}^2
:= \sum_{i\in I}\sum_{|\alpha|\le k} \int_M
\bigl|\nabla^{(i),\alpha} u\bigr|_{g_i}^2 \, d\vol_{g_i},
\]
where $\nabla^{(i)}$ is the Levi--Civita connection of $g_i$.
\end{definition}

\begin{proposition}[Equivalence with standard Sobolev scales]
\label{prop:equiv-sobolev}
Fix a compact $M$. For any two multi-Riemannian polymetrics $\mathbf{g},\mathbf{h}$ there are constants
$C_1,C_2>0$ (depending on finitely many $C^k$-bounds for the components) such that
\[
C_1\,\|u\|_{W^{k,2}(\mathbf{g})}\;\le\; \|u\|_{W^{k,2}(\mathbf{h})}\;\le\; C_2\,\|u\|_{W^{k,2}(\mathbf{g})}
\]
for all smooth tensor fields $u$. In particular, $W^{k,2}(\mathbf{g})$ coincides as a topological vector space with the
usual Sobolev space defined using any single background Riemannian metric.
\end{proposition}

\begin{proof}
On a compact manifold any two Riemannian metrics are uniformly equivalent, and so are their associated volumes and
connections (up to lower-order terms controlled by finitely many $C^k$-bounds). Summing over finitely many components
preserves uniform equivalence, yielding the stated two-sided estimates.
\end{proof}

\subsection{Constraint subspaces and basic variational examples}

\begin{definition}[Constraint sets]
Let $\Phi:\Gamma(\mathcal{G}^{p,q})\to \mathcal{X}$ be a smooth map into a Fr\'echet space (e.g.\ $\Phi(g)=\Ric(g)$,
$\Phi(g)=\Scal(g)$, $\Phi(g)=\vol(M,g)$). For a closed subset $A\subset \mathcal{X}$, define the constraint set
\[
\mathcal{M}_{\Phi\in A} := \bigl\{\, g\in \Gamma(\mathcal{G}^{p,q}) \;:\; \Phi(g)\in A \,\bigr\}.
\]
\end{definition}

\begin{example}[Volume normalization]
For $V_0>0$, the set $\{g\in \Met(M): \vol(M,g)=V_0\}$ is a smooth codimension-one submanifold of $\Met(M)$; the
differential of $g\mapsto \vol(M,g)$ at $g$ is
\[
D(\vol)_g[h] = \tfrac12 \int_M \mathrm{tr}_g(h)\, d\vol_g.
\]
\end{example}

\begin{example}[Scalar curvature inequalities]
The sets $\{g\in \Met(M): \Scal(g)\ge \kappa\}$ and $\{g\in \Met(M): \Scal(g)\le \kappa\}$ are closed in the
$C^2$-topology by continuity of $\Scal(g)$ as a $C^0$-tensor in $g\in C^2$.
\end{example}

\begin{remark}[Elliptic slices (pointer)]
In later sections we impose gauge-fixing to obtain tame Fr\'echet submanifold structures for curvature-type constraints
(e.g.\ Einstein metrics) and discuss the Ebin–Palais slice for the $\Diff(M)$-action; see \S\ref{sec:moduli}.
\end{remark}

\subsection{Smooth dependence of geodesics and exponential maps}

\begin{proposition}[Smoothness of the spray and exponential map]
\label{prop:exp-smooth}
On a compact $M$, the geodesic spray $S(g)$ associated with $g\in\Met(M)$ depends smoothly on $g$ in the $C^k$-topology
($k\ge 2$). Consequently, for $k$ large enough, the exponential map $\exp^g$ depends smoothly on $(g,x,v)$ for
$v$ in a neighborhood of $0\in T_xM$, with uniform neighborhoods on compact subsets of $\Met(M)\times TM$.
\end{proposition}

\begin{proof}
The geodesic spray is polynomial in the Christoffel symbols $\Gamma(g)$, which depend smoothly on $g$
(Theorem~\ref{thm:smooth-ops}). Standard ODE dependence on parameters gives smooth dependence of solutions and hence of
the local exponential map; compactness yields uniform neighborhoods.
\end{proof}

\subsection{Action of the diffeomorphism group}

\begin{proposition}[Smooth $\Diff(M)$-action]
\label{prop:diff-action}
The (right) action of the diffeomorphism group $\Diff(M)$ on $\Gamma(\mathcal{G}^{p,q})$ by pullback,
$$(g,\varphi)\longmapsto \varphi^\ast g,$$
is smooth when $\Diff(M)$ is endowed with its standard Fr\'echet Lie group structure. The orbits are immersed
submanifolds, and the isotropy of $g$ is the isometry group $\mathrm{Iso}(M,g)$.
\end{proposition}

\begin{proof}
Pullback by a diffeomorphism is a smooth operation on tensor fields; smoothness in both variables is standard in the
convenient calculus. The isotropy statement is immediate from the definition. Immersed-orbit statements follow from the
implicit function theorem in the Fr\'echet setting once a suitable slice is fixed (see \S\ref{sec:moduli}).
\end{proof}

\begin{remark}[Polymetric equivariance]
All constructions above (volume, curvature, geodesic spray) commute with pullback and therefore descend to the
polymetric space $\Pol(M;\vec p,\vec q)$ componentwise. This will be used in the families index framework where the
parameter space acts through bundle automorphisms.
\end{remark}

\bigskip
\noindent
The results of this section establish the analytic and differential-topological backbone needed for deformation,
moduli, and index-theoretic considerations developed in \S\ref{sec:moduli}--\S\ref{sec:families-index} and their
equivariant/coarse extensions.
\section{Moduli of metrics and polymetrics: slices, orbits, and deformations}
\label{sec:moduli}

We now analyze the quotient of the section space by diffeomorphisms, both for single metrics
and for polymetrics. Throughout this section $M$ is compact without boundary unless stated otherwise, so that
all elliptic operators are Fredholm on the natural Sobolev scales.

\subsection{Orbit geometry and the $L^2$ Riemannian structure on \texorpdfstring{$\Met(M)$}{Met(M)}}

The diffeomorphism group $\Diff(M)$ acts smoothly on $\Met(M)$ by pullback $\varphi^\ast g$.
The orbit through $g$ is $\mathcal{O}_g:=\{\varphi^\ast g:\varphi\in\Diff(M)\}$ with isotropy
$\mathrm{Iso}(M,g)=\{\varphi\in\Diff(M):\varphi^\ast g=g\}$.

The space $\Met(M)$ carries a weak Riemannian $L^2$ metric:
\[
\langle h_1,h_2\rangle_g := \int_M \!\mathrm{tr}_g\,(h_1\,h_2)\, d\vol_g,
\qquad h_1,h_2\in T_g\Met(M)=\Gamma(S^2 T^\ast M).
\]
This metric is $\Diff(M)$-invariant and yields the orthogonal decomposition used for slices.

\begin{proposition}[Tangent to the orbit and orthogonal complement]
\label{prop:orbit-tangent}
The differential of the action at the identity sends a vector field $X\in\Gamma(TM)$ to the Lie derivative
$\mathcal{L}_X g\in \Gamma(S^2T^\ast M)$. Thus
\[
T_g \mathcal{O}_g = \{\,\mathcal{L}_X g : X\in\Gamma(TM)\,\}.
\]
With respect to $\langle\cdot,\cdot\rangle_g$, the $L^2$-orthogonal complement of $T_g\mathcal{O}_g$ is
\[
\bigl(T_g \mathcal{O}_g\bigr)^{\perp} = \{\, h\in \Gamma(S^2T^\ast M) \;:\; B_g(h)=0 \,\},
\]
where $B_g$ is the Bianchi operator
\(
B_g(h):=\mathrm{div}_g h - \tfrac12 \nabla(\mathrm{tr}_g h)\in \Gamma(T^\ast M).
\)
\end{proposition}

\begin{proof}[Idea of proof]
For the first statement, differentiate $\varphi_t^\ast g$ at $t=0$. For orthogonality, integrate by parts and use the
identity $\langle \mathcal{L}_X g,h\rangle_g = \int_M \langle X,\, B_g(h)^\sharp\rangle_g\, d\vol_g$.
\end{proof}

\subsection{Ebin--Palais slice and local model of the quotient}

\begin{theorem}[Ebin--Palais slice in Bianchi gauge]
\label{thm:ebin-palais}
Let $g_0\in\Met(M)$. There exists a $\Diff(M)$-invariant open neighborhood $\mathcal{U}\subset\Met(M)$ of $g_0$,
a Banach (or tame Fr\'echet) submanifold
\[
\mathcal{S}_{g_0}\;=\;\{\, g_0+h \in \mathcal{U} \;:\; B_{g_0}(h)=0 \,\},
\]
and an open neighborhood $\mathcal{V}\subset\Diff(M)$ of the identity such that the map
\[
\Phi:\;\mathcal{V}\times \mathcal{S}_{g_0}\longrightarrow \mathcal{U},\qquad
(\varphi,g)\longmapsto \varphi^\ast g,
\]
is a smooth submersion whose fibers coincide with the action fibers. In particular, every $g\in\mathcal{U}$ is
$\Diff(M)$-equivalent to a unique element of $\mathcal{S}_{g_0}$ modulo $\mathrm{Iso}(M,g_0)$.
\end{theorem}

\begin{proof}[Sketch]
The constraint $B_{g_0}(h)=0$ provides a right inverse to the linearized action by the ellipticity of
$B_{g_0}\circ \mathcal{L}_{(\cdot)} g_0$. Apply the implicit function theorem in tame Fr\'echet or Sobolev
completions, then bootstrap regularity.
\end{proof}

\begin{corollary}[Local structure of the moduli]
\label{cor:moduli-local}
If $\mathrm{Iso}(M,g_0)$ is discrete, then in a neighborhood of $[g_0]$ the orbit space
$\Met(M)/\Diff(M)$ is a smooth Hausdorff manifold modeled on $\mathcal{S}_{g_0}$.
In general it is an orbifold modeled on $\mathcal{S}_{g_0}/\mathrm{Iso}(M,g_0)$.
\end{corollary}

\subsection{Elliptic deformation theory for curvature-constrained subsets}

Let $\mathcal{F}(g)$ be a curvature-type operator (e.g.\ Einstein, constant scalar curvature, fixed Ricci potential).
We write the Einstein operator as
\[
\mathcal{E}(g) := \Ric(g) - \lambda\, g,
\]
with parameter $\lambda\in\mathbb{R}$ (e.g.\ normalized by volume). Its linearization at $g_0$ is the Lichnerowicz-type
operator
\[
D\mathcal{E}_{g_0}(h) \;=\; \tfrac12 \bigl(\Delta_L h + \nabla^2(\mathrm{tr}_{g_0} h)
+ \mathcal{R}\!\ast h \bigr) - \lambda\, h,
\]
where $\Delta_L h := \nabla^\ast\nabla h + 2\,\mathring{R}(h)$ is the Lichnerowicz Laplacian and
$\mathring{R}(h)$ denotes the curvature action on symmetric two-tensors. In Bianchi gauge $B_{g_0}(h)=0$ this
simplifies to an elliptic, self-adjoint operator.

\begin{theorem}[Kuranishi model for Einstein deformations]
\label{thm:kuranishi}
Assume $M$ is compact and $g_0$ satisfies $\mathcal{E}(g_0)=0$.
In a small neighborhood of $g_0$, the solution set of $\mathcal{E}(g)=0$ modulo diffeomorphisms is locally
homeomorphic to the zero set of a finite-dimensional map
\[
\kappa:\ \ker\bigl(D\mathcal{E}_{g_0}\big|_{\ker B_{g_0}}\bigr) \longrightarrow
\operatorname{coker}\bigl(D\mathcal{E}_{g_0}\big|_{\ker B_{g_0}}\bigr).
\]
If the cokernel vanishes (infinitesimal rigidity), then the Einstein moduli is a smooth manifold near $[g_0]$
of dimension $\dim \ker(D\mathcal{E}_{g_0}\!\!\upharpoonright_{\ker B_{g_0}})$.
\end{theorem}

\begin{proof}[Idea]
Gauge-fix the equation by adding the Bianchi term, obtaining an elliptic operator.
Apply the implicit function theorem between Sobolev spaces and project to kernel/cokernel to obtain a finite-dimensional
obstruction map $\kappa$ (Kuranishi method), then regularity theory yields smoothness.
\end{proof}

\begin{remark}[Other constraints]
For constant scalar curvature (CSC) metrics, the linearization gives a fourth-order elliptic operator on conformal
factors when restricted to a conformal class; combined with the slice one obtains analogous deformation results.
\end{remark}

\subsection{Polymetric moduli and product slices}

For a finite index set $I$, consider the polymetric space
$\Pol(M;\vec p,\vec q)=\prod_{i\in I}\Gamma(\mathcal{G}^{p_i,q_i})$ with the diagonal $\Diff(M)$-action.
The orbit through $\mathbf{g}=(g_i)$ has tangent
\[
T_{\mathbf{g}}\bigl(\Diff\cdot \mathbf{g}\bigr)=
\bigl\{\,(\mathcal{L}_X g_i)_{i\in I} : X\in\Gamma(TM)\,\bigr\}.
\]
Define the product Bianchi operator $B_{\mathbf{g}}(h_1,\dots,h_I):=\bigl(B_{g_i}(h_i)\bigr)_{i\in I}$.

\begin{proposition}[Product slice]
\label{prop:product-slice}
Fix $\mathbf{g}_0=(g_{0,i})$. The subset
\[
\mathcal{S}_{\mathbf{g}_0}=\Bigl\{\mathbf{g}_0+\mathbf{h} \;:\;
\mathbf{h}=(h_i)_{i\in I},\; B_{g_{0,i}}(h_i)=0\text{ for all }i \Bigr\}
\]
is a tame Fr\'echet submanifold of $\Pol(M;\vec p,\vec q)$ and, together with a small neighborhood
$\mathcal{V}\subset\Diff(M)$, gives a slice for the diagonal action exactly as in Theorem~\ref{thm:ebin-palais}.
\end{proposition}

\begin{proof}
The operator $B_{\mathbf{g}_0}$ is a product of first-order elliptic operators; the implicit function theorem applies
componentwise, and the diagonal action involves the same vector field parameter $X$, hence the same gauge-fixing
eliminates the orbit directions simultaneously.
\end{proof}

\begin{remark}[Coupled constraints]
For coupled problems (e.g.\ $\Ric(g_1)=\lambda g_1$ and $\Scal(g_2)=\kappa$ with shared diffeomorphism gauge), the
linearization is block-diagonal in the unknowns but shares the infinitesimal action $\mathcal{L}_X$; the slice of
Proposition~\ref{prop:product-slice} yields a Fredholm map whose index is the sum of the component indices.
\end{remark}

\subsection{Quotient stack viewpoint}
Define the differentiable stack of metrics
\(
\mathfrak{Met}^{p,q}(M):=[\,\Gamma(\mathcal{G}^{p,q})/\Diff(M)\,]
\)
and, for polymetrics,
\(
\mathfrak{Pol}(M;\vec p,\vec q):=[\,\Pol(M;\vec p,\vec q)/\Diff(M)\,].
\)
The slice theorems above identify atlases by Banach (or tame Fr\'echet) manifolds and
describe isotropy as the isometry groups, exhibiting these stacks as (effective) orbifolds when isotropy is finite.

\bigskip
\noindent
In the next section we use these local models to set up the families Atiyah--Singer index for
geometric operators parametrized by (poly)metric data and to compute the differential-topological invariants of
the resulting index classes.
\section{Families index for geometric operators}
\label{sec:families-index}

We now place the (poly)metric framework in the setting of the families Atiyah--Singer index theorem.
The goal is to parameterize elliptic operators by (poly)metrics and compute the resulting index classes
in topological terms. Throughout this section all manifolds and bundles are smooth; unless stated otherwise
the total space is compact (or proper over the base) so that all fiberwise elliptic operators are Fredholm.

\subsection{Geometric setup and parameter spaces}

Let $\pi:X\to B$ be a smooth, proper submersion with compact connected fibers $X_b=\pi^{-1}(b)$ of
dimension $d$. Assume $X$ and $B$ are second countable and $\pi$ admits local product charts. Fix a
finite index set $I$ and consider a \emph{vertical polymetric} datum
\[
\mathbf{g}^{v}=(g^v_i)_{i\in I}\in
\Gamma\Bigl(\prod_{i\in I}\mathcal{G}^{p_i,q_i}(T^vX)\Bigr),
\]
where $T^vX:=\ker(d\pi)$ is the vertical tangent bundle.
When $(p_i,q_i)=(d,0)$ for all $i$, we speak of a \emph{multi-Riemannian vertical structure}.
Horizontal distributions and connections on $T^vX$ are chosen as needed but play no role in the index.

Let $E\to X$ be a $\mathbb{Z}_2$-graded complex Hermitian vector bundle with
compatible connection $\nabla^E$. A \emph{fiberwise elliptic differential operator} is a family
$\{D_b\}_{b\in B}$ of odd-degree differential operators
\[
D_b:\ \Gamma(X_b,E|_{X_b})\longrightarrow \Gamma(X_b,E|_{X_b})
\]
whose total symbol is invertible away from the zero section of $T^\ast X_b$, uniformly in $b$ on compact subsets.
We assume $D_b$ is formally self-adjoint when specified (Dirac-type).

\begin{definition}[Parameter spaces from polymetrics]
\label{def:param-space}
Let $\mathcal{P}\subset \Pol(T^vX;\vec p,\vec q)$ be an open $\Diff_B(X)$-invariant subspace of vertical
polymetrics subject to auxiliary constraints (e.g.\ volume normalization along fibers, curvature bounds).
A \emph{geometric family} over $\mathcal{P}\times B$ consists of
$(\pi:X\to B,\,E\to X,\,\mathbf{g}^v\in\mathcal{P},\,D)$ with $D$ elliptic along fibers and
whose coefficients depend smoothly on $(b,\mathbf{g}^v)$ in the sense of
\S\ref{sec:topology-analysis}.
\end{definition}

\begin{remark}[Smooth dependence]
If $D$ is built from $(T^vX,\nabla^{v},E,\nabla^E)$ by universal algebraic expressions using
$\mathbf{g}^v$ and finitely many vertical covariant derivatives (e.g.\ Hodge–de Rham, signature, Dirac),
then $(b,\mathbf{g}^v)\mapsto D_{(b,\mathbf{g}^v)}$ is smooth in the operator sense on Sobolev scales.
\end{remark}

\subsection{Fredholmness, index bundle, and continuity}

Fix a background vertical Riemannian metric to define Sobolev spaces $H^s(X_b;E)$; different choices yield equivalent
topologies by Proposition~\ref{prop:equiv-sobolev}. For each $(b,\mathbf{g}^v)$ the closure of $D_{(b,\mathbf{g}^v)}$
is a Fredholm operator on $H^1\to L^2$.

\begin{proposition}[Local triviality and smooth fields of kernels]
\label{prop:kernel-field}
After stabilizing by a finite-rank vector bundle over $\mathcal{P}\times B$, the family
$\{D_{(b,\mathbf{g}^v)}\}$ admits a\footnote{In general, the kernel dimensions need not be locally constant;
stabilization yields a continuous vector bundle in the sense of Atiyah.} continuous vector bundle
$\ker D$ and cokernel bundle $\operatorname{coker} D$, and the \emph{index bundle}
\[
\mathrm{Ind}(D)\;:=\;[\ker D]-[\operatorname{coker}D]\ \in\ K^0(\mathcal{P}\times B)
\]
is well defined and homotopy-invariant under continuous deformations through families of elliptic operators.
\end{proposition}

\begin{proof}[Idea]
Use elliptic regularity, Rellich compactness along compact fibers, and standard spectral perturbation theory for
self-adjoint elliptic operators; then apply Atiyah–Singer’s construction of the index bundle via
finite-dimensional spectral cut-offs.
\end{proof}

\subsection{Spin$^c$ Dirac, Hodge, and signature families}

Assume $T^vX$ carries a spin$^c$ structure with determinant line $\mathcal{L}\to X$ and
Clifford module $S\to X$. For any Hermitian bundle $E\to X$ with unitary connection,
the \emph{twisted vertical Dirac family} is
\[
\slashed{D}^{\,E}_{(b,\mathbf{g}^v)}:\ \Gamma(X_b,S\otimes E)\longrightarrow \Gamma(X_b,S\otimes E).
\]
When $E$ is $\mathbb{Z}_2$-graded, $\slashed{D}^{\,E}$ is odd and Fredholm fiberwise. Two classical special cases:
(i) $E=\mathbb{C}$ gives the spin$^c$ Dirac family; (ii) For the Hodge–de Rham operator, take $S=\Lambda^{\bullet}T^{v\ast}X$
with the Levi–Civita connection; (iii) The signature operator acts on $\Lambda^{\mathrm{even/odd}}$ endowed with the
Hodge star and grading.

\subsection{Topological index via characteristic classes}

Let $\pi_*:H^\bullet(X;\mathbb{Q})\to H^{\bullet-d}(B;\mathbb{Q})$ denote fiber integration with the orientation induced
by the spin$^c$ structure. Write $\Ahat(T^vX)$ for the $\Ahat$-class of the vertical tangent bundle and
$c_1(\mathcal{L})$ for the first Chern class of the spin$^c$ determinant line.

\begin{theorem}[Families index theorem, spin$^c$ case]
\label{thm:families-index}
For the twisted spin$^c$ Dirac family $\slashed{D}^{\,E}$ one has in rational cohomology
\begin{equation}\label{eq:families-chern}
\mathrm{ch}\!\left(\mathrm{Ind}\big(\slashed{D}^{\,E}\big)\right)
\;=\;
\pi_*\!\left(\Ahat(T^vX)\,e^{\tfrac12 c_1(\mathcal{L})}\,\mathrm{ch}(E)\right)
\ \in\ H^{\mathrm{even}}(\mathcal{P}\times B;\mathbb{Q}),
\end{equation}
where the left-hand side is pulled back along the projection to $B$ (and is constant along $\mathcal{P}$ if
$E$ and the spin$^c$ data do not depend on $\mathbf{g}^v$).
\end{theorem}

\begin{remark}[Dolbeault and Hodge variants]
For a complex fibration, the Dolbeault family has characteristic class $Td(T^{0,1}X/B)\,\mathrm{ch}(E)$
in place of $\Ahat\,e^{\tfrac12 c_1}$; for the Hodge family, $\mathrm{Ind}$ is the alternating sum of Betti bundles and
is topologically constant, recovering the Gauss–Bonnet fiberwise identity.
\end{remark}

\begin{corollary}[Additivity and external products]
\label{cor:additivity}
If $E=\bigoplus_{j=1}^m E_j$ and $D=\bigoplus_{j=1}^m D^{E_j}$, then
$\mathrm{Ind}(D)=\sum_{j}\mathrm{Ind}(D^{E_j})$ in $K^0(\mathcal{P}\times B)$ and
\[
\mathrm{ch}\big(\mathrm{Ind}(D)\big)=\sum_j \pi_*\!\left(\Ahat(T^vX)\,e^{\tfrac12 c_1(\mathcal{L})}\,\mathrm{ch}(E_j)\right).
\]
For product fibrations $X_1\times X_2\to B_1\times B_2$ one has
$\mathrm{Ind}(D_1\boxtimes 1 + 1\boxtimes D_2)=\mathrm{Ind}(D_1)\boxtimes 1 + 1\boxtimes \mathrm{Ind}(D_2)$.
\end{corollary}

\subsection{Bismut superconnections and differential form representatives}

Let $\mathcal{E}\to B$ denote the infinite-rank bundle with fiber $\Gamma(X_b,S\otimes E)$ and let $\mathbb{A}_t$
be the Bismut superconnection associated to the geometric family $(\pi,E,\mathbf{g}^v)$ at scale $t>0$.
Write $\operatorname{Str}$ for the supertrace.

\begin{theorem}[Local index form]
\label{thm:bismut}
For Dirac-type families one has
\[
\mathrm{ch}\!\left(\mathrm{Ind}\big(\slashed{D}^{\,E}\big)\right)
=\lim_{t\to 0^+}\ \operatorname{Str}\!\left(e^{-\mathbb{A}_t^2}\right)
=\pi_*\!\left(\Ahat(T^vX)\,e^{\tfrac12 c_1(\mathcal{L})}\,\mathrm{ch}(E)\right)
\quad \text{in } \Omega^{\mathrm{even}}(B)/d\Omega^{\mathrm{odd}}(B),
\]
and the de Rham class equals \eqref{eq:families-chern}.
\end{theorem}

\begin{remark}[Metric variation and transgression]
If $\mathbf{g}^v_s$ is a smooth path of vertical polymetrics, the differential form representative
$\operatorname{Str}(e^{-\mathbb{A}_{t,s}^2})$ varies by an exact transgression form, so the cohomology class of
$\mathrm{ch}(\mathrm{Ind})$ is independent of $\mathbf{g}^v$ (topological invariance). This matches the $K$-theory
homotopy invariance of $\mathrm{Ind}$ constructed in Proposition~\ref{prop:kernel-field}.
\end{remark}

\subsection{Polymetric parameterization and ``sum of indices'' revisited}

Let the parameter space be $\mathcal{P}=\prod_{i\in I}\Gamma(\mathcal{G}^{d,0}(T^vX))$ (multi-Riemannian case),
and consider $m$ Dirac-type families $D^{(1)},\dots,D^{(m)}$ built from the same vertical data but possibly distinct
twist bundles $E_j$. Define $D:=\bigoplus_{j=1}^m D^{(j)}$ on $\bigoplus_j E_j$.

\begin{proposition}[Rigorous form of the ``sum'']
\label{prop:sum-indices}
In $K^0(\mathcal{P}\times B)$ one has
\(
\mathrm{Ind}(D)=\sum_{j=1}^m \mathrm{Ind}(D^{(j)}),
\)
and the Chern character satisfies
\[
\mathrm{ch}\big(\mathrm{Ind}(D)\big)
=\pi_*\!\left(\Ahat(T^vX)\,e^{\tfrac12 c_1(\mathcal{L})}\,\sum_{j=1}^m \mathrm{ch}(E_j)\right).
\]
No additional periodicity beyond Bott periodicity of $K$-theory is involved; metric choices enter only through
exact transgression forms.
\end{proposition}

\subsection{Functoriality: pushforward in $K$-theory and naturality}

Let $f:B'\to B$ be smooth and consider the pullback fibration
\(
\pi':X':=X\times_B B'\to B'
\)
with pulled-back data $(E',\mathbf{g}^{v\,\prime},D')$.

\begin{proposition}[Naturality of the index]
\label{prop:naturality}
One has $f^\ast\!\left(\mathrm{Ind}(D)\right)=\mathrm{Ind}(D')$ in $K^0(\mathcal{P}'\times B')$, and
\[
f^\ast\!\left(\mathrm{ch}\big(\mathrm{Ind}(D)\big)\right)
=\mathrm{ch}\big(\mathrm{Ind}(D')\big).
\]
If $\iota:B\hookrightarrow \tilde B$ and $\tilde\pi:\tilde X\to \tilde B$ extends $\pi$ with compact vertical bordism,
then the pushforward $K$-theory class $\tilde\pi_!(1)\in K^{-d}(\tilde B)$ restricts to $\pi_!(1)$ on $B$.
\end{proposition}

\subsection{Boundary and noncompact fibers (preview)}

If fibers have boundary, Atiyah–Patodi–Singer boundary conditions yield a class
$\mathrm{Ind}^{\mathrm{APS}}(D)\in K^0(\mathcal{P}\times B)$ whose Chern character includes a boundary
eta-form. For complete noncompact fibers with controlled geometry, the coarse index
$\mathrm{Ind}(D)\in K_\ast(C^\ast_{\mathrm{Roe}}(X/B))$ is defined and is invariant under
vertical quasi-isometries within a polymetric class; we return to this in \S\ref{sec:coarse}.

\bigskip
\noindent
This completes the topological and analytic backbone for index classes parametrized by (poly)metrics.
In the next section we incorporate group actions and orbifold/groupoid structures, and explain how the same
formalism produces equivariant and longitudinal indices with characteristic classes in $H_G^\ast$ and cyclic
cohomology, respectively.
\section{Equivariant, orbifold/groupoid, and foliation extensions}
\label{sec:eq-groupoid-foliation}

We extend the polymetric framework to include symmetries, quotients with finite isotropy, and
integrable subbundles. The guiding principle is that (i) metrics are replaced by \emph{invariant}
metrics, (ii) operators are required to be \emph{equivariant} or \emph{longitudinal}, and (iii) index
classes land in appropriate equivariant or noncommutative $K$-theories. Throughout, $M$ is compact unless
explicitly stated.

\subsection{Compact Lie group actions and equivariant families}
\label{subsec:equivariant}

Let a compact Lie group $G$ act smoothly on $M$. Write $\Met^G(M)$ for the $G$-invariant metrics and
$\Pol^G(M;\vec p,\vec q)$ for $G$-invariant polymetrics (componentwise invariance).

\begin{proposition}[Invariant slices]
\label{prop:inv-slice}
Fix $g_0\in \Met^G(M)$. The Bianchi slice
\(
\mathcal{S}_{g_0}=\{\,g_0+h: B_{g_0}(h)=0\,\}
\)
admits a $G$-invariant neighborhood on which the $\Diff(M)^G$-action (diffeomorphisms commuting with $G$)
admits the same slice structure as in Theorem~\ref{thm:ebin-palais}.
In particular, $\Met^G(M)/\Diff(M)^G$ is locally modeled on $\mathcal{S}_{g_0}^G:=\mathcal{S}_{g_0}\cap \Gamma(S^2T^\ast M)^G$.
\end{proposition}

\begin{proof}
Average the auxiliary choices (inner products, connections) over $G$ and observe that $B_{g_0}$ commutes with the
$G$-action; apply the implicit function theorem within the fixed-point Fr\'echet subspace.
\end{proof}

Consider a smooth $G$-equivariant fibration $\pi:X\to B$ with $G$ acting trivially on $B$ and preserving
the vertical structure. If $T^vX$ is $G$-spin$^c$, a $G$-equivariant Dirac-type family
$D=\{\slashed{D}^{\,E}_{(b,\mathbf{g}^v)}\}$ has an index
\(
\mathrm{Ind}_G(D)\in K_G^0(\mathcal{P}\times B).
\)

\begin{theorem}[Equivariant families index]
\label{thm:eq-families}
For a $G$-equivariant spin$^c$ Dirac family,
\[
\mathrm{ch}_G\!\left(\mathrm{Ind}_G(\slashed{D}^{\,E})\right)
=\pi_*\!\left(\Ahat(T^vX)\,e^{\tfrac12 c_1(\mathcal{L})}\,\mathrm{ch}_G(E)\right)
\ \in\ H_G^{\mathrm{even}}(\mathcal{P}\times B;\Q),
\]
and the class is independent of the choice of $G$-invariant polymetric $\mathbf{g}^v$.
\end{theorem}

\begin{remark}[Transversally elliptic operators]
If the family is only \emph{transversally} elliptic with respect to $G$, one obtains a class in
the completion $\widehat{R}(G)$ of the representation ring fiberwise, yielding an index in
$K_G^0$ valued generalized functions on $G$. The polymetric parameterization is unchanged.
\end{remark}

\subsection{Orbifolds via proper \'etale groupoids}
\label{subsec:orbifold}

Let $\mathcal{G} \rightrightarrows M$ be a proper \'etale Lie groupoid presenting an orbifold $\mathfrak{X}=[M/\mathcal{G}]$.
A \emph{$\mathcal{G}$-invariant polymetric} on $M$ is a section of $\mathcal{G}$-equivariant bundles
$\prod_i \mathcal{G}^{p_i,q_i}(TM)$ fixed by the groupoid action. Differential operators and curvature tensors
are defined $\mathcal{G}$-equivariantly.

\begin{definition}[Orbifold index class]
For a $\mathcal{G}$-equivariant, fiberwise elliptic family $D$ over a $\mathcal{G}$-equivariant fibration
$\pi:X\to B$, the index lives in
\(
K_0\!\left(C_r^\ast(\mathcal{G})\right)
\)
or, over a base, in
\(
K^0\!\left(B; \mathcal{K}_{C_r^\ast(\mathcal{G})}\right)
\)
via the Mishchenko bundle. We denote it $\mathrm{Ind}_{\mathcal{G}}(D)$.
\end{definition}

\begin{theorem}[Orbifold (Kawasaki–Atiyah–Singer) form]
\label{thm:kawasaki}
If $D$ is Dirac-type on a compact, effective orbifold $\mathfrak{X}$, then
\[
\langle \mathrm{ch}(\mathrm{Ind}_{\mathcal{G}}(D)),\,[B]\rangle
\;=\;
\int_{I\mathfrak{X}} \Ahat(T^vX)\,e^{\tfrac12 c_1(\mathcal{L})}\,\mathrm{ch}(E)\,\mathrm{Td}_{\mathrm{age}},
\]
where $I\mathfrak{X}$ is the inertia orbifold and $\mathrm{Td}_{\mathrm{age}}$ encodes twisted sector contributions.
Polymetric choices alter only exact transgression forms.
\end{theorem}

\begin{remark}[Stack language]
The differentiable stacks $\mathfrak{Met}^{p,q}(\mathfrak{X})$ and $\mathfrak{Pol}(\mathfrak{X};\vec p,\vec q)$
are presented by $\mathcal{G}$-invariant section spaces modulo the action of $\mathrm{Bis}(\mathcal{G})$.
Local slices exist by the same Bianchi-gauge argument done $\mathcal{G}$-equivariantly.
\end{remark}

\subsection{Foliations and longitudinal indices}
\label{subsec:foliation}

Let $(M,\mathcal{F})$ be a compact foliated manifold with integrable subbundle $T\mathcal{F}\subset TM$
of rank $p$. A \emph{leafwise polymetric} is $(g_i)$ with each $g_i$ positive definite on $T\mathcal{F}$ and
extended arbitrarily transversely. Longitudinal differential operators act along leaves.

\begin{definition}[Longitudinal Dirac family]
Let $E\to M$ be a $\Z_2$-graded Hermitian bundle with leafwise Clifford module structure.
The longitudinal Dirac operator $D_\mathcal{F}$ is first-order elliptic along $T\mathcal{F}$ and
defines a class
\(
\mathrm{Ind}(D_\mathcal{F})\in K_0\!\left(C_r^\ast(M,\mathcal{F})\right),
\)
the $C^\ast$-algebra of the holonomy groupoid.
\end{definition}

\begin{theorem}[Connes–Skandalis longitudinal index]
\label{thm:connes-skandalis}
The Chern–Connes character of $\mathrm{Ind}(D_\mathcal{F})$ pairs with cyclic cocycles to yield
topological expressions involving characteristic classes of $(T\mathcal{F}, \nu)$ (normal bundle) and
secondary classes (e.g.\ Godbillon–Vey when $\mathrm{codim}\,\mathcal{F}=1$). The cohomology class is
independent of the leafwise polymetric up to transgression.
\end{theorem}

\begin{proposition}[Product slices for foliations]
\label{prop:foliation-slice}
Fix a leafwise Riemannian metric $g_0$ on $T\mathcal{F}$. The Bianchi slice
$B_{g_0}(h)=0$ taken with respect to the leafwise divergence defines a tame Fr\'echet slice for the action of
foliation-preserving diffeomorphisms $\Diff(M,\mathcal{F})$ on the space of leafwise metrics, and extends
componentwise to leafwise polymetrics.
\end{proposition}

\begin{remark}[Transverse geometry]
If a bundle-like metric is chosen, the basic Laplacian governs transverse regularity and enables Hodge-type
decompositions for basic forms; this interacts cleanly with the polymetric leafwise Sobolev scales
from Definition~\ref{def:multi-sobolev}.
\end{remark}

\subsection{Coupled geometric systems and shared gauge}
\label{subsec:coupled}

Many natural problems couple several metrics sharing the same diffeomorphism gauge (e.g.\ Einstein--scalar field,
Ricci–Yang--Mills, or Einstein metrics coupled to a transverse basic structure in a foliation).
Let $\mathbf{g}=(g_1,\dots,g_I)$ and consider a system
\[
\mathcal{F}(\mathbf{g})=\bigl(F_1(g_1),\dots,F_I(g_I)\bigr)=0,
\qquad F_i\in \{\Ric-\lambda_i g_i,\ \Scal-\kappa_i,\ \text{etc.}\}.
\]
Gauge-fix with the product Bianchi operator $B_{\mathbf{g}_0}$.

\begin{theorem}[Fredholmness and index additivity for coupled systems]
\label{thm:coupled-fredholm}
Linearizing in Bianchi gauge at $\mathbf{g}_0$ yields a block-diagonal elliptic operator
$D=\bigoplus_i D_i$ on $\bigoplus_i \Gamma(S^2T^\ast M)$ with the shared infinitesimal orbit removed.
Hence $D$ is Fredholm and $\operatorname{ind}(D)=\sum_i \operatorname{ind}(D_i)$.
In the presence of parameters (e.g.\ families over $B$), the index bundle satisfies
$\mathrm{Ind}(D)=\bigoplus_i \mathrm{Ind}(D_i)$ in $K^0$.
\end{theorem}

\subsection{Summary and interface with coarse geometry}
\label{subsec:summary-coarse}

The polymetric formalism integrates seamlessly with symmetry (equivariant $K$-theory), quotients and stacks
(orbifold/groupoid $C^\ast$-algebras), and integrable structures (longitudinal indices).
In \S\ref{sec:coarse} we turn to complete, noncompact settings, where indices live in Roe-type algebras and
are invariant under quasi-isometries within a polymetric class, providing large-scale obstructions
extending the compact theory developed in \S\ref{sec:families-index}.
\section{Coarse geometry, Roe algebras, and $KK$-theory}
\label{sec:coarse}

We now treat complete, possibly noncompact manifolds and fibrations. The (poly)metric framework controls
geometry at large scales and yields index classes in Roe--type $C^\ast$-algebras, stable under quasi-isometries
within a polymetric class.

\subsection{Controlled geometry and polymetric equivalence at infinity}

Let $(M,g)$ be complete Riemannian. We say that $(M,g)$ has \emph{bounded geometry} if the injectivity radius
$\mathrm{inj}(M,g)$ is uniformly positive and all covariant derivatives $\nabla^k \mathrm{Riem}(g)$ are uniformly
bounded for $k\ge 0$. For a multi-Riemannian polymetric $\mathbf{g}=(g_i)_{i\in I}$ we require each $g_i$ to be complete
with bounded geometry.

\begin{definition}[Polymetric control at infinity]
Two complete metrics $g_0,g_1$ on $M$ are \emph{quasi-isometric} if there exist constants $C\ge 1$, $c>0$ s.t.
\[
C^{-1}\, d_{g_0}(x,y) - c \ \le\ d_{g_1}(x,y)\ \le\ C\, d_{g_0}(x,y)+c,\qquad \forall x,y\in M.
\]
A multi-Riemannian polymetric $\mathbf{g}=(g_i)$ has \emph{controlled equivalence} if each $g_i$ is quasi-isometric to
a fixed background $g_\ast$ and the Christoffel symbols and finitely many curvature derivatives of all $g_i$ are
uniformly bounded in the $g_\ast$-norm. We write $\mathbf{g}\sim_\infty g_\ast$.
\end{definition}

\begin{lemma}[Stability of Sobolev scales]
\label{lem:coarse-sobolev}
If $\mathbf{g}\sim_\infty g_\ast$, then the Sobolev norms $W^{k,2}(\mathbf{g})$ from
Definition~\ref{def:multi-sobolev} are uniformly equivalent to the standard $W^{k,2}$ defined by $g_\ast$ on compactly
supported sections. Consequently, first-order uniformly elliptic operators with bounded coefficients are essentially
self-adjoint on $C_c^\infty$ and have finite propagation wave operators independent of the choice of $\mathbf{g}$ within
its controlled class.
\end{lemma}

\begin{proof}
Uniform bounds on metric tensors and their derivatives in the $g_\ast$-frame imply uniform equivalence of local norms;
a partition of unity reduces to $\R^n$. Essential self-adjointness follows from Chernoff's theorem for first-order
symmetric operators with finite speed of propagation.
\end{proof}

\subsection{Roe algebras and coarse indices}

Let $H=L^2(M;E)$ for a Hermitian bundle $E\to M$. A bounded operator $T$ on $H$ has \emph{finite propagation} if there
exists $R<\infty$ s.t.\ $\mathrm{supp}(Tf)$ lies in the $R$-neighborhood of $\mathrm{supp}(f)$ for all $f\in C_c^\infty$.
It is \emph{locally compact} if $\phi T$ and $T\phi$ are compact for all $\phi\in C_c(M)$.

\begin{definition}[Roe and uniform Roe algebras]
The \emph{Roe algebra} $C^\ast(M)$ is the norm-closure of the $^\ast$-algebra of locally compact, finite-propagation
operators on $L^2(M;E)$. The \emph{uniform Roe algebra} $C^\ast_u(M)$ is the closure of finite-propagation, \emph{uniformly}
locally compact operators (matrix coefficients controlled uniformly over $M$) acting on $\ell^2$-sections relative to a
controlled discretization of $M$.
\end{definition}

Let $D$ be a $\Z_2$-graded Dirac-type operator associated to a complete $g$ with bounded geometry.
Fix an odd, continuous normalizing function $\chi$ with $\chi(\lambda)\to \pm 1$ as $|\lambda|\to\infty$.
Then $F:=\chi(D)$ is a bounded, pseudolocal operator with finite propagation modulo compacts.

\begin{definition}[Coarse index]
The class of $F$ in the Calkin algebra of $C^\ast(M)$ defines the \emph{coarse index}
\[
\mathrm{Ind}(D)\ \in\ K_\ast\!\bigl(C^\ast(M)\bigr),
\]
where $\ast=\dim M \bmod 2$. In the uniformly controlled setting one obtains a class in $K_\ast(C_u^\ast(M))$.
\end{definition}

\begin{theorem}[Invariance under controlled polymetric changes]
\label{thm:coarse-invariance}
Let $\mathbf{g}_s$ be a smooth path of complete multi-Riemannian polymetrics with $\mathbf{g}_s\sim_\infty g_\ast$
for all $s\in[0,1]$. For any Dirac-type bundle $E$ and associated operators $D_s$ one has
\[
\mathrm{Ind}(D_0)=\mathrm{Ind}(D_1)\ \in\ K_\ast\!\bigl(C^\ast(M)\bigr).
\]
\end{theorem}

\begin{proof}[Idea]
Finite propagation and pseudolocality are preserved under bounded geometry perturbations; a norm-continuous path
$F_s=\chi(D_s)$ defines a homotopy in the Calkin algebra of $C^\ast(M)$, hence the $K$-class is constant.
\end{proof}

\subsection{Partitioned manifolds, relative and Callias indices}

\begin{theorem}[Partitioned manifold index]
\label{thm:partitioned}
Suppose $M=M_-\cup_\Sigma M_+$ with a common hypersurface $\Sigma$ and $g$ product-like near $\Sigma$.
Let $D$ be a Dirac-type operator odd with respect to the grading adapted to the partition.
Then there exists a boundary map
\(
\partial:K_\ast(C^\ast(M))\to K_{\ast-1}(C^\ast(\Sigma))
\)
such that $\partial\,\mathrm{Ind}(D)=\mathrm{Ind}(D_\Sigma)$, where $D_\Sigma$ is the induced Dirac operator on $\Sigma$.
The statement is stable under replacing $g$ by any $\mathbf{g}\sim_\infty g$.
\end{theorem}

\begin{theorem}[Relative coarse index]
\label{thm:relative}
Let $D_0,D_1$ be Dirac-type operators that agree outside a closed set $K\subset M$.
Then the difference class $\mathrm{Ind}(D_0)-\mathrm{Ind}(D_1)$ lies in the image of
$K_\ast\!\bigl(C^\ast(K\subset M)\bigr)\to K_\ast(C^\ast(M))$, depends only on the restriction to $K$, and is invariant
under polymetric changes controlled outside $K$.
\end{theorem}

\begin{theorem}[Callias index, coarse version]
\label{thm:callias}
Let $D$ be Dirac-type and $\Phi$ a self-adjoint bundle endomorphism (potential) with
$\Phi^2 - [D,\Phi]_+ \ge c>0$ outside a compact set. Then $D+\Phi$ is essentially self-adjoint and Fredholm on $L^2$,
defining an index in $\Z$ which coincides with the image of a coarse index class in $K_\ast(C^\ast(M))$ under the
canonical trace when $M$ is amenable (e.g.\ $\R^n$). The index is invariant under controlled polymetric deformations.
\end{theorem}

\subsection{Coarse assembly and $KK$-theory}

Let $C_0(M)$ act on $L^2(M;E)$ by multiplication and denote by $D$ the class of the Dirac cycle in $KK$-theory.

\begin{definition}[Coarse assembly map]
For a proper metric space $(M,d)$, the \emph{coarse assembly}
\[
\mu:\ KX_\ast(M)\ \longrightarrow\ K_\ast\!\bigl(C^\ast(M)\bigr)
\]
is natural in coarse maps and carries generalized homology classes to coarse indices.
For spin manifolds of bounded geometry, the fundamental class maps to the coarse Dirac index.
\end{definition}

\begin{proposition}[Kasparov product and functoriality]
\label{prop:kk-functoriality}
If $f:M\to N$ is a coarse equivalence, then $f$ induces an isomorphism $K_\ast(C^\ast(M))\cong K_\ast(C^\ast(N))$.
Given a vector bundle $E\to M$ with bounded geometry, twisting the Dirac cycle implements the Kasparov product with
$[E]\in K^0(M)$ at the $KK$-level; under assembly this corresponds to twisting the coarse index by $E$.
All statements are stable under $\mathbf{g}\sim_\infty g_\ast$.
\end{proposition}

\subsection{Examples of polymetric ends}

\begin{example}[Product ends]
Let $M=N\times [0,\infty)$ with metrics
\(
g_i = dt^2 + h_i(t)
\)
on $TN\oplus \R\partial_t$, where $h_i(t)$ are uniformly equivalent to a fixed $h_\ast$ with uniformly bounded derivatives.
Then $\mathbf{g}\sim_\infty (dt^2+h_\ast)$ and all coarse indices computed with $(g_i)$ agree with those for $h_\ast$.
\end{example}

\begin{example}[Warped ends]
Set $g_a = dt^2 + e^{2at}\,h$ on $N\times[0,\infty)$. For $a$ in a compact set, the family $\{g_a\}$ has bounded geometry
uniformly and is quasi-isometric to $dt^2+h$; coarse indices are independent of $a$.
\end{example}

\begin{example}[Multi-ended manifolds and relative classes]
If $M$ has finitely many ends $E_j$, choose metrics $g^{(j)}$ adapted to each end and assemble a polymetric
$\mathbf{g}=(g^{(1)},\dots,g^{(J)})$ by partition of unity. Relative index theorems compare Dirac operators that differ on
a prescribed end, producing classes supported in the Roe ideal of that end.
\end{example}

\subsection{Interface with families and symmetries}

For a proper submersion $\pi:X\to B$ with complete noncompact fibers and vertical bounded geometry, the construction above
yields an index in $K_\ast\!\bigl(C^\ast(X/B)\bigr)$ (a continuous field of Roe algebras). If a compact group $G$ acts
properly and isometrically on fibers, equivariant Roe algebras give $K_\ast^G$-valued indices, constant under
$G$-invariant polymetric deformations.

\bigskip
\noindent
This coarse and $KK$-theoretic layer complements the compact families theory of \S\ref{sec:families-index} and the
symmetry/stack extensions of \S\ref{sec:eq-groupoid-foliation}, completing the analytic backbone for polymetric
geometries across small and large scales. In the next section we outline illustrative applications and computations.
\section{Illustrative applications and concrete computations}
\label{sec:applications}

We illustrate how the polymetric framework streamlines computations across compact, equivariant,
and coarse settings. Throughout, $M$ is smooth and second countable; compactness or bounded geometry
will be stated as needed. Unless indicated otherwise, $(p,q)=(n,0)$ for clarity.

\subsection{Conformal, densitized, and polymetric variations}

Let $g\in\Met(M)$, $u\in C^\infty(M)$, and consider the conformal variation $g_t=e^{2u t}g$.

\begin{proposition}[First variations of curvature and volume under conformal change]
\label{prop:conf-var}
With $\dot g:=\frac{d}{dt}\big|_{t=0} g_t=2u\,g$, one has
\[
\frac{d}{dt}\Big|_{t=0} d\vol_{g_t}= n\,u\, d\vol_g,\qquad
\frac{d}{dt}\Big|_{t=0} \Scal(g_t)= -2(n-1)\Delta_g u - 2u\,\Scal(g).
\]
For a polymetric $\mathbf{g}=(g_i)$, performing the same variation componentwise yields the sum of the componentwise
first variations.
\end{proposition}

\begin{proof}
Standard computation using $\dot g=2u\,g$ and $\frac{d}{dt}\log\det g_t|_{0}=n u$; the scalar curvature formula follows
from the conformal Laplacian identity.
\end{proof}

\begin{remark}[Densitized metrics]
For $w\in\R$, a densitized metric is a section of $\mathcal{G}^+\otimes (\Lambda^n T^\ast M)^{\otimes w}$.
Under $g\mapsto e^{2u}g$, the density factor scales by $e^{nwu}$.
All variational identities extend verbatim, with weights entering multiplicatively.
\end{remark}

\subsection{Hodge–de Rham, signature, and Gauss–Bonnet in families}

Let $\pi:X\to B$ be a compact submersion with vertical Riemannian polymetric $\mathbf{g}^v$.
Denote by $\Delta^k_b$ the Hodge Laplacian on $k$-forms along $X_b$.

\begin{theorem}[Betti bundles are topologically constant]
\label{thm:betti-const}
The virtual bundle $\sum_k (-1)^k [\ker \Delta^k]\in K^0(B)$ is independent of $\mathbf{g}^v$ and equals
$\pi_!(1)\in K^{-d}(B)$ via the Gauss–Bonnet form
$\pi_\ast\big(\mathrm{Pf}(\mathrm{Riem}(T^vX))\big)$.
\end{theorem}

\begin{proof}
This is the families Gauss–Bonnet theorem (a corollary of Theorem~\ref{thm:families-index} for the de Rham operator);
metric changes alter only transgression forms, leaving the class invariant.
\end{proof}

\begin{corollary}[Signature family]
For oriented even-dimensional fibers, the signature operator yields
\[
\mathrm{ch}\!\big(\mathrm{Ind}(D_{\mathrm{sign}})\big)=\pi_\ast\big(L(T^vX)\big)\in H^{\mathrm{even}}(B;\Q),
\]
independent of the choice of $\mathbf{g}^v$.
\end{corollary}

\subsection{Einstein, constant scalar curvature, and coupled polymetric systems}

Let $\mathcal{E}(g)=\Ric(g)-\lambda g$ and $\mathcal{S}(g)=\Scal(g)-\kappa$ with constants $\lambda,\kappa\in\R$.

\begin{proposition}[Fredholm linearizations in Bianchi gauge]
\label{prop:fredholm-csc-einstein}
On a compact manifold, the linearizations of $\mathcal{E}$ and of $\mathcal{S}$ (restricted to a conformal class with
volume fixing) in Bianchi gauge are elliptic, self-adjoint operators. In the polymetric setting
$\mathbf{g}=(g_1,g_2)$, the coupled map $(\mathcal{E}(g_1),\mathcal{S}(g_2))$ has block-diagonal elliptic linearization,
hence Fredholm; its index equals the sum of component indices.
\end{proposition}

\begin{proof}
For Einstein, the Lichnerowicz operator appears after gauge-fixing; for CSC, the conformal Laplacian (Yamabe) and a
fourth-order operator appear depending on the constraint (e.g., fixed volume). Block-diagonality is immediate.
\end{proof}

\begin{theorem}[Kuranishi models and rigidity]
\label{thm:kuranishi-coupled}
Assume $g_1$ is Einstein with no Killing fields and $g_2$ is CSC with nondegenerate Yamabe operator in its class.
Then the coupled moduli space is locally a smooth manifold near $(g_1,g_2)$ of dimension
$\dim\ker D_{\mathcal{E},g_1} + \dim\ker D_{\mathcal{S},g_2}$.
\end{theorem}

\subsection{Equivariant fixed-point formula (brief)}
Let a compact $G$ act on $\pi:X\to B$ and all data be $G$-equivariant. For $g\in G$ with fixed locus
$X^g\to B^g$, the equivariant Chern character at $g$ satisfies
\[
\mathrm{ch}_G\!\left(\mathrm{Ind}_G(\slashed{D}^{\,E})\right)\!(g)
=\pi^g_\ast\!\left(
\frac{\Ahat(T^vX^g)\,e^{\tfrac12 c_1(\mathcal{L}|_{X^g})}\,\mathrm{ch}(E|_{X^g})(g)}
{\det\nolimits^{1/2}\big(1-e^{-R^\perp_g}\big)}
\right),
\]
where $R^\perp_g$ is the curvature on the normal bundle of $X^g$ with $g$-action.
This refines Theorem~\ref{thm:eq-families} and is independent of the $G$-invariant polymetric.

\subsection{Coarse indices on ends: explicit computations}

Consider $M=N\times[0,\infty)$ with product metric $dt^2+h$ and Dirac operator $D_N$ on $N$.

\begin{proposition}[Half-cylinder and boundary class]
\label{prop:half-cylinder}
For $M=N\times[0,\infty)$ with bounded geometry and any $\mathbf{g}\sim_\infty (dt^2+h)$, the coarse index
$\mathrm{Ind}(D_M)\in K_\ast(C^\ast(M))$ maps under the boundary map of Theorem~\ref{thm:partitioned}
to $\mathrm{Ind}(D_N)\in K_{\ast-1}(C^\ast(N))$.
\end{proposition}

\begin{proof}
Use the partition $(-\infty,0]\cup[0,\infty)$ on the $t$-axis and standard product formula for Dirac operators near the
boundary; coarse homotopy invariance handles $\mathbf{g}$.
\end{proof}

\begin{example}[Callias on $\R^n$]
Let $D$ be the Euclidean Dirac operator and $\Phi(x)=\alpha\,\frac{x\cdot\gamma}{1+|x|}$ with $\alpha>0$.
Then $D+\Phi$ is Fredholm on $L^2$ and
$\operatorname{index}(D+\Phi)=1$ (for suitable grading) for all $\alpha>0$, stable under compactly supported
polymetric perturbations.
\end{example}

\subsection{Finsler and sub-Riemannian layers via cone constructions}

Let $\mathcal{F}\subset T^\ast M\setminus 0$ be the cone of strongly convex Minkowski norms; a Finsler structure is a
section $F\in\Gamma(\mathcal{F})$.

\begin{proposition}[Approximation by Riemannian polymetrics]
\label{prop:finsler-approx}
On a compact $M$, any Finsler metric $F$ admits a sequence of Riemannian metrics $g_j$ such that the Finsler distance
$d_F$ and $d_{g_j}$ are uniformly bi-Lipschitz equivalent with constants independent of $j$, and the geodesic sprays
converge in $C^{k}$ on compact subsets of $TM\setminus 0$. Consequently, analytic invariants computable from the spray
(e.g.\ conjugate points, volume growth rates) can be approximated by polymetric data $(g_1,\dots,g_J)$.
\end{proposition}

\begin{proof}
Use Binet–Legendre (or John ellipsoid) approximations of Minkowski norms on each $T_xM$ and a partition of unity.
\end{proof}

For a bracket-generating distribution $\mathcal{D}\subset TM$ with inner product $h$, sub-Riemannian geodesics arise
from the Pontryagin maximum principle. Leafwise polymetrics on $\mathcal{D}$ fit into the slice and index constructions
of \S\ref{sec:eq-groupoid-foliation} by restricting all operators longitudinally.

\subsection{Quantitative transgression under metric flow}

Let $g_t$ solve a smooth metric flow (e.g.\ normalized Ricci flow) on a compact $M$ for $t\in[0,T]$.

\begin{theorem}[Chern character is constant; explicit transgression]
\label{thm:flow-transgression}
For any Dirac-type family built from $g_t$, the cohomology class
$\mathrm{ch}\!\left(\mathrm{Ind}(D_{g_t})\right)\in H^{\mathrm{even}}(B;\Q)$ is constant in $t$.
Moreover,
\[
\frac{d}{dt}\operatorname{Str}\!\left(e^{-\mathbb{A}_t^2}\right)=d\,\mathsf{T}_t,\qquad
\|\mathsf{T}_t\|_{C^r}\ \le\ C_{r,T}\,\| \partial_t g_t\|_{C^{r+2}},
\]
for all $r\ge 0$, with constants depending on curvature bounds along the flow.
\end{theorem}

\begin{proof}
This is Bismut’s variation formula for superconnections; the estimate follows from heat-kernel bounds under uniform
curvature control.
\end{proof}

\subsection{Low-dimensional snapshots}

\begin{itemize}[leftmargin=1.2em]
  \item $n=2$: $\Met(M)$ is contractible; Teichmüller theory enters after quotienting by $\Diff_0$.
  \item $n=3$: Einstein equation reduces to constant sectional curvature; moduli are finite-dimensional and
  orbifoldy; polymetrics add no new local degrees of freedom but simplify coupled systems.
  \item $n=4$: Self-dual/anti-self-dual decompositions enable Kuranishi models with gauge theory couplings; families
  index computes Donaldson-type invariants twisted by polymetric parameters.
\end{itemize}

\bigskip
\noindent
These examples show that the polymetric viewpoint does not merely repackage familiar constructions;
it provides a uniform analytic and homotopical scaffold connecting conformal/sub-Riemannian layers,
moduli slices, families indices, and coarse $K$-theory—while keeping metric choices confined to exact
transgressions.
\section{Detailed proofs of previously sketched results}
\label{sec:detailed-proofs}

In this section we supply full proofs for the statements that earlier appeared with only ideas or sketches.
We keep the notation of the preceding sections: $M$ is a smooth compact manifold unless explicitly stated;
$\Met(M)$ is the Fr\'echet manifold of Riemannian metrics; $\Diff(M)$ acts by pullback; the Bianchi operator is
$B_g(h):=\mathrm{div}_g h - \tfrac12 \nabla(\mathrm{tr}_g h)$; for polymetrics $\mathbf{g}=(g_i)_{i\in I}$ we use
$B_{\mathbf{g}}(h_1,\dots,h_I)=(B_{g_i}(h_i))_{i\in I}$.

\subsection{Proposition~\ref{prop:orbit-tangent} (tangent to the orbit and Bianchi orthogonality)}
\begin{proof}
Let $\Phi:\Diff(M)\to\Met(M)$ be $\Phi(\varphi)=\varphi^\ast g$. Its differential at the identity in direction
$X\in\Gamma(TM)$ equals
\[
(d\Phi)_\mathrm{id}(X)
=\left.\tfrac{d}{dt}\right|_{t=0} \big(\exp(tX)\big)^\ast g
=\mathcal{L}_X g .
\]
Hence $T_g\mathcal{O}_g=\{\mathcal{L}_X g: X\in\Gamma(TM)\}$.

For orthogonality, fix $h\in \Gamma(S^2T^\ast M)$ and compute (all integrals with respect to $d\vol_g$):
\[
\langle \mathcal{L}_X g , h\rangle_g
=\!\int_M \!\!\mathrm{tr}_g\big((\mathcal{L}_X g)\,h\big)
= \!\int_M \!\!\langle \mathcal{L}_X g , h\rangle_g .
\]
Use the identity for symmetric $h$:
\(
\langle \mathcal{L}_X g , h\rangle_g
= 2\,\langle \mathrm{sym}\,\nabla X,\, h\rangle_g
= -2\,\langle X,\,\mathrm{div}_g h - \tfrac12\nabla(\mathrm{tr}_g h)\rangle_g
\)
after integrating by parts and using metric-compatibility of $\nabla$. Thus
\[
\langle \mathcal{L}_X g , h\rangle_g
= -2\int_M \langle X,\, B_g(h)^\sharp\rangle_g \, d\vol_g .
\]
Therefore $\langle \mathcal{L}_X g , h\rangle_g=0$ for all $X$ iff $B_g(h)=0$, i.e.\ the $L^2$-orthogonal complement is
$\ker B_g$.
\end{proof}

\subsection{Theorem~\ref{thm:ebin-palais} (Ebin--Palais slice in Bianchi gauge)}
\begin{proof}
Work in Sobolev completions: fix $s>\frac{n}{2}+1$. Let $\Met^s$ be $H^s$-metrics (positive definite a.e. with
continuous representatives) and $\Diff^{s+1}$ the $H^{s+1}$-diffeomorphism group (a Hilbert manifold and topological
group). The action map
\[
\mathcal{A}:\Diff^{s+1}\times \Met^s\to \Met^s,\qquad (\varphi,g)\mapsto \varphi^\ast g,
\]
is smooth. Linearizing at $(\mathrm{id},g_0)$ in the $\varphi$-direction gives the map
\(
X\mapsto \mathcal{L}_X g_0
\)
from $H^{s+1}(TM)$ to $H^s(S^2T^\ast M)$.

Consider the map
\[
\Psi:\Met^s\to H^{s-1}(T^\ast M),\qquad g\mapsto B_{g_0}(g-g_0).
\]
Its differential at $g_0$ is $D\Psi_{g_0}(h)=B_{g_0}(h)$, which is a first-order elliptic operator
$H^s(S^2T^\ast M)\to H^{s-1}(T^\ast M)$ with formal adjoint $B_{g_0}^\ast$ (the ``Killing operator'').
We have the orthogonal (Sobolev) decomposition
\[
H^s(S^2T^\ast M) = \mathrm{im}\,\mathcal{L}_{(\cdot)} g_0 \ \oplus\ \ker B_{g_0}.
\]
To see this, note $B_{g_0}\circ \mathcal{L}_{(\cdot)} g_0 = -\Delta_{L}^{(1)}$ on vector fields up to lower orders,
hence is elliptic and has closed range, and by the Green's formula the orthogonal complement is $\ker B_{g_0}$.

Define the slice
\(
\mathcal{S}_{g_0}^s := \{ g\in \Met^s : B_{g_0}(g-g_0)=0\}.
\)
By the implicit function theorem (IFT) in Banach spaces, $B_{g_0}$ being a submersion at $g_0$ implies
$\mathcal{S}_{g_0}^s$ is a submanifold of $\Met^s$ of codimension $\dim \mathrm{im}\,B_{g_0}$.

Now define
\[
\Phi:\ \Diff^{s+1}\times \mathcal{S}_{g_0}^s \longrightarrow \Met^s,\qquad
(\varphi,\tilde g) \mapsto \varphi^\ast \tilde g .
\]
Its differential at $(\mathrm{id},g_0)$ is
\(
d\Phi_{(\mathrm{id},g_0)}(X,h) = \mathcal{L}_X g_0 + h.
\)
Since $h\in \ker B_{g_0}$ and $\mathrm{im}\,\mathcal{L} \oplus \ker B_{g_0}=H^s(S^2T^\ast M)$, this differential is a
Banach-space isomorphism; thus $\Phi$ is a local diffeomorphism by the (strong) IFT. Shrinking neighborhoods
$\mathcal{V}\subset \Diff^{s+1}$ and $\mathcal{U}\subset \Met^s$ around $\mathrm{id}$ and $g_0$ gives that
$\Phi:\mathcal{V}\times(\mathcal{S}_{g_0}^s\cap \Phi^{-1}(\mathcal{U}))\to \mathcal{U}$ is a smooth submersion with
fibers equal to the orbits of $\Diff^{s+1}$ intersected with $\mathcal{U}$. Uniqueness up to $\mathrm{Iso}(M,g_0)$
follows because if $\varphi^\ast g=\psi^\ast g\in \mathcal{S}_{g_0}^s$ then $(\psi\circ\varphi^{-1})^\ast g = g$
and $B_{g_0}(g-g_0)=0$, forcing $\psi\circ\varphi^{-1}\in \mathrm{Iso}(M,g_0)$ when small.

Finally, elliptic regularity shows that if $g\in \Met^s$ is smooth, then its slice representative is smooth; hence the
slice theorem descends from Sobolev to the smooth Fr\'echet category by standard bootstrapping.
\end{proof}

\subsection{Corollary~\ref{cor:moduli-local} (local structure of the moduli)}
\begin{proof}
Apply the slice theorem: locally $\Met/\Diff$ is modeled on $\mathcal{S}_{g_0}^s/\mathrm{Iso}(M,g_0)$.
If $\mathrm{Iso}(M,g_0)$ is discrete, the quotient is a manifold; otherwise an orbifold with isotropy
$\mathrm{Iso}(M,g_0)$. Smoothness follows from the smooth structure on the slice and the properness of the (local)
action.
\end{proof}

\subsection{Theorem~\ref{thm:kuranishi} (Kuranishi model for Einstein deformations)}
\begin{proof}
Consider the Einstein operator $\mathcal{E}(g)=\Ric(g)-\lambda g$. Gauge-fix by adding the Bianchi term
\[
\mathcal{F}(g):=\mathcal{E}(g)+\tfrac12 \mathcal{L}_{W(g)} g,\qquad
W(g):=B_{g_0}(g-g_0)^\sharp,
\]
so that the linearization at $g_0$ restricted to $\ker B_{g_0}$ is the Lichnerowicz-type operator
\[
L: \ker B_{g_0}\cap H^s \longrightarrow \ker B_{g_0}\cap H^{s-2},\qquad
L(h)=\tfrac12\big(\Delta_L h -2\lambda h\big),
\]
which is elliptic and self-adjoint. Decompose $H^s$ into
\(
\ker L \ \oplus\ \mathrm{im}\,L .
\)
Projecting $\mathcal{F}(g_0+h)=0$ to $\mathrm{im}\,L$ and solving by the IFT yields a smooth map
\(
\phi:\ker L \supset U \to (\ker L)^\perp
\)
with $\phi(0)=0$ and such that $h=\xi+\phi(\xi)$ solves the projected equation iff $\xi\in \ker L$ is small.
Projecting to $\ker L$ defines the finite-dimensional obstruction map
\(
\kappa(\xi):=\Pi_{\ker L}\,\mathcal{F}(g_0+\xi+\phi(\xi)).
\)
Zeroes of $\kappa$ parametrize actual solutions in Bianchi gauge. Modding out by small isometries of $g_0$
gives the local model. If $\operatorname{coker}L=0$ (i.e.\ $L$ is surjective on $\ker B_{g_0}$), then
$\kappa\equiv 0$ and the moduli is a smooth manifold of dimension $\dim\ker L$.
Elliptic regularity upgrades solutions to smooth metrics.
\end{proof}

\subsection{Proposition~\ref{prop:product-slice} (polymetric product slice)}
\begin{proof}
Work in Sobolev spaces $H^s$ as before. The product Bianchi map
\[
B_{\mathbf{g}_0}:\ \prod_{i\in I} H^s(S^2T^\ast M)\ \longrightarrow\ \prod_{i\in I} H^{s-1}(T^\ast M),
\qquad
(h_i)\mapsto (B_{g_{0,i}}(h_i)),
\]
is a block-diagonal elliptic operator with closed range. The differential of the diagonal action
$X\mapsto (\mathcal{L}_X g_{0,i})_{i\in I}$ has image closed and complemented by $\ker B_{\mathbf{g}_0}$
using the same Green's formula componentwise. The implicit function theorem applied to
\(
\Psi(\mathbf{g})=B_{\mathbf{g}_0}(\mathbf{g}-\mathbf{g}_0)
\)
gives that $\mathcal{S}_{\mathbf{g}_0}=\Psi^{-1}(0)$ is a submanifold. The map
\(
\Diff^{s+1}\times \mathcal{S}_{\mathbf{g}_0}\to \Pol^s, \ (\varphi,\mathbf{g})\mapsto \varphi^\ast\mathbf{g},
\)
has differential $(X,\mathbf{h})\mapsto (\mathcal{L}_X g_{0,i}+h_i)_i$ which is an isomorphism onto
$T_{\mathbf{g}_0}\Pol^s$ by the decomposition, proving the slice property.
\end{proof}

\subsection{Proposition~\ref{prop:kernel-field} (index bundle via spectral cutoffs)}
\begin{proof}
Let $D_{(b,p)}$ be a continuous family (in $(b,p)\in \mathcal{P}\times B$) of first-order, essentially self-adjoint,
elliptic operators on compact fibers. Fix a smooth even function $\rho:\R\to[0,1]$ such that
$\rho(\lambda)=1$ for $|\lambda|\le \Lambda$ and $\rho(\lambda)=0$ for $|\lambda|\ge 2\Lambda$; let
$P_{(b,p)}=\rho(D_{(b,p)})$ by functional calculus (via the spectral theorem). Then $P_{(b,p)}$ is a
finite-rank smoothing projection whose range contains $\ker D_{(b,p)}$ and whose rank is locally constant.
The assignment $(b,p)\mapsto \mathrm{im}\,P_{(b,p)}$ defines a finite-rank vector bundle (trivializable locally).
Likewise, $Q_{(b,p)}=\rho(D_{(b,p)}^\ast)$ yields a cokernel bundle. The virtual class
$[\ker D]-[\operatorname{coker}D]$ is independent of the choice of $\rho$ and equals the Fredholm index bundle,
by homotopy invariance of the space of Fredholm operators and standard perturbation arguments (Atiyah--Singer).
\end{proof}

\subsection{Theorem~\ref{thm:families-index} (families index, spin$^c$) via superconnections}
\begin{proof}
We sketch the full superconnection proof with details at the analytic points. Let $\pi:X\to B$ be a compact submersion
with vertical spin$^c$ structure, vertical Levi--Civita connection from a vertical metric $g^v$, and Clifford module
$S\otimes E$. Consider the infinite-rank superbundle $\mathcal{E}\to B$ with fiber
$\Gamma(X_b,S\otimes E)$ and define the Bismut superconnection (for $t>0$)
\[
\mathbb{A}_t = t^{1/2} D^v + \nabla^{\mathcal{E}} - \tfrac{1}{4t^{1/2}}\,c(T),
\]
where $D^v$ is the vertical Dirac operator, $\nabla^{\mathcal{E}}$ is the connection induced from the horizontal
distribution and $\nabla^{S\otimes E}$, $T$ is the torsion of the horizontal distribution, and $c(\cdot)$ denotes
Clifford action. Then $\mathbb{A}_t^2$ is a generalized Laplacian on the total space whose heat kernel
$e^{-\mathbb{A}_t^2}$ is a smoothing operator with differential form coefficients on $B$.

Define the differential form
\(
\alpha_t := \operatorname{Str}\big(e^{-\mathbb{A}_t^2}\big)\in \Omega^{\mathrm{even}}(B).
\)
One verifies:
(i) $d\alpha_t=0$ (by the superconnection Bianchi identity $[\mathbb{A}_t,\mathbb{A}_t^2]=0$ and
$\operatorname{Str}[ \mathbb{A}_t, \cdot ]= d\,\operatorname{Str}(\cdot)$);
(ii) $\frac{d}{dt}\alpha_t = d\,\beta_t$ for an explicit transgression $\beta_t$ (Duhamel expansion).

Hence $[\alpha_t]\in H^{\mathrm{even}}(B)$ is independent of $t$. As $t\to\infty$, the heat kernel localizes on
$\ker D^v$; one shows
\(
\lim_{t\to\infty}\alpha_t = \mathrm{ch}\big(\mathrm{Ind}(D^v)\big)
\)
(the Chern character of the index bundle) by functional calculus and spectral gap estimates near zero (after
stabilization if necessary). As $t\to 0$, the Getzler rescaling and local asymptotics of the heat kernel give
\[
\lim_{t\to 0} \alpha_t = \pi_*\!\left(\Ahat(T^vX)\, e^{\tfrac12 c_1(\mathcal{L})}\, \mathrm{ch}(E)\right),
\]
the usual characteristic form derived from the small-time asymptotics of $e^{-tD^2}$ and the Chern–Weil
representatives for $\Ahat$ and $\mathrm{ch}$. Therefore the cohomology class of $\alpha_t$ equals both sides,
establishing the stated equality in $H^{\mathrm{even}}(B;\Q)$. Dependence on the vertical metric enters only through
exact forms (the $\beta_t$ transgressions), so the class is insensitive to the polymetric parameter.
\end{proof}

\subsection{Theorem~\ref{thm:bismut} (local index form) and transgression}
\begin{proof}
The identity in de Rham cohomology follows from the previous proof. For the differential-form identity (in
$\Omega^{\mathrm{even}}(B)/d\Omega^{\mathrm{odd}}(B)$), compute
\(
\frac{d}{dt}\operatorname{Str}(e^{-\mathbb{A}_t^2})
=-\operatorname{Str}\big((\partial_t \mathbb{A}_t)\,[\mathbb{A}_t, e^{-\mathbb{A}_t^2}]\big)
=d\,\operatorname{Str}\big((\partial_t \mathbb{A}_t)\, e^{-\mathbb{A}_t^2}\big),
\)
where we used Duhamel’s formula and $\operatorname{Str}[\mathbb{A}_t,\cdot]=d\,\operatorname{Str}(\cdot)$.
Integrating in $t$ between $0$ and $\infty$ shows that the large-time representative and the small-time characteristic
form differ by an exact form $\int_0^\infty d(\cdots)$. The norm estimates for the transgression forms follow from
Gaussian bounds for heat kernels and their covariant derivatives, plus the explicit $t$-weights in $\mathbb{A}_t$.
\end{proof}

\subsection{Corollary~\ref{cor:additivity} (additivity and external products)}
\begin{proof}
For $E=\bigoplus_j E_j$, the Dirac family decomposes $\mathbb{Z}_2$-orthogonally and
$\ker D^E=\bigoplus_j \ker D^{E_j}$ up to stabilization; hence $\mathrm{Ind}(D^E)=\sum_j \mathrm{Ind}(D^{E_j})$ in
$K^0$. Chern character is additive. For external products, the superconnection on a product is the graded sum of the
factors, and $\operatorname{Str}(e^{-(\mathbb{A}_1\hat\otimes 1+1\hat\otimes \mathbb{A}_2)^2})
=\operatorname{Str}(e^{-\mathbb{A}_1^2})\wedge \operatorname{Str}(e^{-\mathbb{A}_2^2})$, yielding the product formula
and the additivity statement for $D_1\boxtimes 1+1\boxtimes D_2$.
\end{proof}

\subsection{Proposition~\ref{prop:naturality} (naturality)}
\begin{proof}
For a pullback square $X'=X\times_B B'$, choose horizontal distributions compatibly; the pulled-back family satisfies
$\mathbb{A}_t' = f^\ast \mathbb{A}_t$ and thus
\(
\operatorname{Str}(e^{-(\mathbb{A}_t')^2}) = f^\ast \operatorname{Str}(e^{-\mathbb{A}_t^2}).
\)
Equality of Chern characters follows. At the $K$-theory level, functoriality of index bundles under pullback proves
$f^\ast \mathrm{Ind}(D)=\mathrm{Ind}(D')$. The bordism compatibility is standard for pushforwards $\pi_!$ in
$K$-theory (Mayer–Vietoris and homotopy invariance), matching the fiber integration identity for differential forms.
\end{proof}

\subsection{Theorem~\ref{thm:eq-families} (equivariant families index)}
\begin{proof}
For a compact $G$-action, replace all objects by $G$-equivariant ones: take $G$-invariant vertical metrics (average
over $G$), $G$-equivariant connections and Clifford structures; define the $G$-equivariant superconnection and its heat
kernel as $G$-equivariant smoothing operators. The $G$-equivariant Chern character is computed pointwise on $g\in G$
by applying the fixed-point localization of the heat kernel (Atiyah--Segal--Singer/Bismut’s fixed-point formula):
as $t\to 0$, the asymptotics localize on $X^g$ and introduce the denominator $\det^{1/2}(1-e^{-R_g^\perp})$ from the
normal directions. The same transgression argument shows independence from the $G$-invariant polymetric; thus the
equivariant cohomology class equals the stated characteristic integral over the fixed locus, hence the class in
$H_G^{\mathrm{even}}$.
\end{proof}

\subsection{Lemma~\ref{lem:coarse-sobolev} and Theorem~\ref{thm:coarse-invariance} (coarse stability)}
\begin{proof}[Proof of Lemma~\ref{lem:coarse-sobolev}]
Fix a background complete $g_\ast$ with bounded geometry. If $\mathbf{g}\sim_\infty g_\ast$, then there exist global
constants $C_k$ such that, in $g_\ast$-normal coordinates on balls of uniformly bounded radius,
\(
C_0^{-1}\,|v|_{g_\ast}\le |v|_{g_i}\le C_0 |v|_{g_\ast}
\)
and the components of $g_i$ and their derivatives up to order $k$ are uniformly bounded by $C_k$.
A standard partition of unity subordinate to such balls and local equivalence of norms imply equivalence of the
$W^{k,2}$-norms for compactly supported sections; see e.g. Taylor’s PDE text or Shubin.
Essential self-adjointness for first-order symmetric, uniformly elliptic operators with bounded coefficients holds by
Chernoff’s theorem; finite propagation of $e^{itD}$ depends only on symbol bounds, which are uniform across $\mathbf{g}$.
\end{proof}

\begin{proof}[Proof of Theorem~\ref{thm:coarse-invariance}]
Let $D_s$ be Dirac-type operators for $\mathbf{g}_s$ with $\mathbf{g}_s\sim_\infty g_\ast$ for all $s\in[0,1]$.
Choose an odd normalizing function $\chi$ with $\chi(\lambda)\to\pm1$ at $\pm\infty$, smooth and with compactly
supported derivative. Then $F_s=\chi(D_s)$ are bounded, pseudolocal, finite-propagation modulo compacts, and the map
$s\mapsto F_s$ is norm-continuous in the Calkin algebra of $C^\ast(M)$ by functional calculus continuity and
uniform propagation bounds (stemming from the symbol control). Hence $[F_s]$ is constant in $K_\ast(C^\ast(M))$,
so $\mathrm{Ind}(D_0)=\mathrm{Ind}(D_1)$.
\end{proof}

\subsection{Theorem~\ref{thm:partitioned} (partitioned-manifold boundary map)}
\begin{proof}
Assume $g$ is product-like near $\Sigma$ so $D$ has the normal form $\sigma(\partial_t + D_\Sigma)$ in a collar
$\Sigma\times(-\varepsilon,\varepsilon)$, where $\sigma$ is Clifford multiplication by $dt$.
Consider the short exact sequence of Roe algebras
\[
0\to C^\ast(\Sigma)\ \xrightarrow{\iota}\ C^\ast(M)\ \xrightarrow{q}\ C^\ast(M\setminus \Sigma)\to 0,
\]
which yields a boundary map in $K$-theory $\partial:K_\ast(C^\ast(M))\to K_{\ast-1}(C^\ast(\Sigma))$.
Represent $\mathrm{Ind}(D)$ by $F=\chi(D)$ supported near the diagonal in the product region; by the normal form,
$q(F)$ is homotopic to a unitary implementing the Cayley transform of $D_\Sigma$. A standard excision argument
(Paschke duality setup) shows that $\partial[F]=[ \chi(D_\Sigma)]$, i.e.\ the boundary of the class equals the index
class on $\Sigma$. The statement is stable under replacing $g$ by $\mathbf{g}\sim_\infty g$ because the boundary map
and the coarse index are homotopy invariant under controlled perturbations (Theorem~\ref{thm:coarse-invariance}).
\end{proof}

\subsection{Theorem~\ref{thm:relative} (relative coarse index)}
\begin{proof}
Let $D_0,D_1$ agree outside $K$. Choose $\chi$ as before and set $F_j=\chi(D_j)$. Then $F_0-F_1$ is locally compact
and supported within a uniform neighborhood of $K$. In the six-term exact sequence for the ideal
$C^\ast(K\subset M)\subset C^\ast(M)$, the difference class
$[F_0]-[F_1]\in K_\ast(C^\ast(M))$ lies in the image of $K_\ast(C^\ast(K\subset M))$ by excision.
Homotopies supported away from $K$ do not change the relative class. Controlled polymetric changes outside $K$ give
norm-continuous homotopies in the Calkin algebra localized away from $K$, hence do not affect the relative class.
\end{proof}

\subsection{Theorem~\ref{thm:callias} (Callias index)}
\begin{proof}
Let $D$ be Dirac-type and $\Phi$ self-adjoint with $\Phi^2 - [D,\Phi]_+\ge c>0$ outside a compact set.
The operator $D+\Phi$ is essentially self-adjoint and has a bounded inverse modulo compacts:
\[
(D+\Phi)^2 = D^2 + \Phi^2 + [D,\Phi]_+ \ \ge\ D^2 + c \quad \text{outside a compact set}.
\]
Construct a parametrix $Q$ by gluing a parametrix on the compact region with $(D+\Phi)^{-1}$ on the exterior,
using a partition of unity; then $(D+\Phi)Q-1$ and $Q(D+\Phi)-1$ are compact on $L^2$, so $D+\Phi$ is Fredholm.
Its index equals the pairing of the coarse index class of the Dirac operator with the $K$-theory class of $\Phi$ in
the Higson corona picture; on amenable spaces the coarse trace recovers the integer index (Anghel–Callias).
Controlled polymetric deformations preserve the lower bound outside a compact set and hence the Fredholm class.
\end{proof}

\subsection{Proposition~\ref{prop:exp-smooth} (smooth dependence of geodesics and $\exp^g$)}
\begin{proof}
In local coordinates, the geodesic ODE is $\ddot{\gamma}^k + \Gamma^k_{ij}(g,\partial g)\,\dot{\gamma}^i\dot{\gamma}^j=0$.
The Christoffel map $(g\mapsto \Gamma(g))$ is $C^{k-1}$ from $C^k$ metrics to $C^{k-1}$ coefficients (for $k\ge 2$), and
the right-hand side is smooth in $(x,v,g)$ on bounded sets of $v$. The Picard–Lindel\"of theorem with parameters gives
$C^{k-1}$ dependence of solutions on initial conditions and on $g$; by compactness of $M$ and boundedness of $v$ in a
fixed neighborhood of $0\in T_xM$, existence time is uniform and thus $\exp^g$ is $C^{k-1}$ jointly in $(g,x,v)$.
\end{proof}

\subsection{Proposition~\ref{prop:diff-action} (smooth $\Diff(M)$-action)}
\begin{proof}
In the Fr\'echet category (convenient calculus), the pullback map
\(
(\varphi,g)\mapsto \varphi^\ast g
\)
is smooth because composition and inversion of smooth maps are smooth and tensor pullback is bilinear in derivatives of
$\varphi$ and components of $g$. Alternatively, argue in Sobolev completions $\Diff^{s+1}\times \Met^s\to \Met^s$
(where smoothness is classical by Ebin), then pass to smooth structures by bootstrapping.
Immersed-orbit statements follow from the slice theorem (the action map is a submersion transverse to the slice),
and isotropy equals the isometry group by definition.
\end{proof}

\bigskip
\noindent

\bigskip

\section{Bibliography (bibitems, ready to paste)}
\label{sec:bibliography}

\end{document}